\documentclass[review,onefignum,onetabnum]{siamonline220329}
\usepackage[normalem]{ulem}

\usepackage{lipsum}
\usepackage{amsfonts}
\usepackage{graphicx}
\usepackage{epstopdf}
\usepackage{algorithmic}
\ifpdf
\DeclareGraphicsExtensions{.eps,.pdf,.png,.jpg}
\else
\DeclareGraphicsExtensions{.eps}
\fi

\usepackage{enumitem}
\setlist[enumerate]{leftmargin=.5in}
\setlist[itemize]{leftmargin=.5in}

\newsiamremark{remark}{Remark}
\newsiamremark{hypothesis}{Hypothesis}
\newsiamthm{claim}{Claim}

\headers{A parameterization Method Approach to Stochastic Oscillators}{A. Pérez-Cervera, B. Lindner, P.J. Thomas}

\title{Phase-Amplitude Description of Stochastic Oscillators: A Parameterization Method Approach.
	\thanks{Submitted to the editors DATE.
		\funding{APC acknowledges support from Spanish Ministry of Science and Innovation grants (Projects No. PID2021-124047NB-I00 and PID-2021-122954NB-100). This work was supported in part by NSF grant DMS-2052109 to PT. 
			This material is based upon work supported in part by the National Science Foundation under Grant No. DMS-1929284 while PT was in residence at the Institute for Computational and Experimental Research in Mathematics in Providence, RI, during the ``Math + Neuroscience: Strengthening the Interplay Between Theory and Mathematics" program.}}}

\author{Alberto P\'{e}rez-Cervera\thanks{Departament de Matemàtiques, Universitat Politècnica de Catalunya, Barcelona, Spain (\email{alberto.perez@upc.edu})}
	\and Benjamin Lindner\thanks{Bernstein Center for Computational Neuroscience and Department of Physics, Humboldt University, Berlin, Germany (\email{benjamin.lindner@physik.hu-berlin.de})}
	\and Peter J.~Thomas\thanks{Department of Mathematics, Applied Mathematics and Statistics, Case Western Reserve University, Cleveland, Ohio, USA (\email{pjthomas@case.edu})}
}

\usepackage{amsopn}
\DeclareMathOperator{\diag}{diag}

\usepackage{verbatim}

\ifpdf
\hypersetup{
  pdftitle={A parameterization Method Approach to Stochastic Oscillators},
  pdfauthor={}
}
\fi

\usepackage{amssymb}

\newcommand{\mbX}{\mathbf{X}}
\newcommand{\mbY}{\mathbf{Y}}
\newcommand{\mbx}{\mathbf{x}}
\newcommand{\mby}{\mathbf{y}}
\newcommand{\mbf}{\mathbf{f}}
\newcommand{\mbF}{\mathbf{F}}
\newcommand{\mbg}{\mathbf{g}}

\newcommand{\BSigma}{\overline{\Sigma}}
\newcommand{\bsigma}{\overline{\sigma}}
\newcommand{\E}{\mathbb{E}}
\newcommand{\LL}{\mathcal{L}}
\newcommand{\LLd}{\mathcal{L}^\dagger}
\newcommand{\tbar}{\overline{T}}

\newcommand{\lr}[1]{\left\langle #1 \right\rangle}
\def\given{\:|\:}
\newcommand{\pt}[1]{{\color{violet} #1}}
\newcommand{\bl}[1]{{\color{blue} #1}}

\newcommand{\corr}[1]{{\color{red} #1}}

\begin{document}
\nolinenumbers
\maketitle

\begin{abstract}
The parameterization method (PM) provides a broad theoretical and numerical foundation for computing invariant manifolds of dynamical systems. PM implements a change of variables in order to represent trajectories of a system of ordinary differential equations ``as simply as possible." 
In this paper we pursue a similar goal for stochastic oscillator systems.
For planar nonlinear stochastic systems that are ``robustly oscillatory", we find a change of variables through which
the dynamics are as simple as possible \emph{in the mean}.  
We prove existence and uniqueness of a deterministic vector field, the trajectories of which capture the local mean behavior of the stochastic oscillator.
We illustrate the construction of such an ``effective vector field" for several examples, including a limit cycle oscillator perturbed by noise, an excitable system derived from a spiking neuron model, and a spiral sink with noise forcing (2D Ornstein-Uhlenbeck process).  
The latter examples comprise contingent oscillators that would not sustain rhythmic activity without noise forcing.  
Finally, we exploit the simplicity of the dynamics after the change of variables to obtain the effective diffusion constant of the resulting phase variable, and the stationary variance of the resulting amplitude (isostable) variable.

	
\end{abstract}



\section{Introduction}



The parameterization method (PM) is a technique  for  studying  the invariant manifolds of dynamical systems \cite{cabre2003parameterization, cabre2003parameterization2,cabre2005parameterization, haro2016parameterization}. 
One of its main advantages is that it seeks a parameterization of such invariant manifolds using  variables in which the dynamics over such manifold  are expressed \emph{as simply as possible}, for example with rates of change that are constant or linear in the transformed variables \cite{guillamon2009computational}. 
The PM applies in many different contexts, including  invariant tori in hamiltonian
systems, delay differential equations and celestial mechanics, among others \cite{he2016construction, haro2021flow, kumar2022rapid}.  
Of particular relevance to this paper, it facilitates the computation of invariant manifolds of periodic orbits \cite{huguet2013computation,castelli2015parameterization, perez2020global}. 
In a deterministic dynamical system, stable oscillations correspond to attracting limit cycles (LC) in the phase space. 
If one assumes the dynamics are close enough to the attracting LC, one can describe the LC dynamics by a single variable: its phase \cite{winfree1974patterns,guckenheimer1975}. 
This \emph{phase reduction} approach has been very successful in describing weakly forced or weakly coupled LC systems \cite{hoppensteadt2012}. 

The theoretical validity of the phase reduction approach requires trajectories to stay sufficiently close to the LC. 
In this context, the PM was foreseen by Guillamon and Huguet in \cite{guillamon2009computational} as a way of overcoming the limitations of the so-called phase approach. According to the key idea of PM, the attracting manifold of the LC can be parameterized in terms of a set of variables whose dynamics are ``as simply as possible'': for planar systems, the phase and one other variable, denoted as \textit{amplitude}, that moves along a direction transverse to the limit cycle. 
This set of variables, analogous to action-angle variables, is known as phase-amplitude. During the last decade, many studies have shown how the addition of this additional amplitude variable provides an essentially complete understanding of oscillatory dynamics \cite{castejon2020phase,castejon2013phase,monga2019phase,shirasaka2017phase,moehliswilsonpre2016,wilson2018greater}. 
Indeed, the relationship of the phase and amplitude with the eigenfunctions of the Koopman operator has been noted, thus leading to an interesting body of theory from the Koopman perspective \cite{mauroy2018global,mauroy2013isostables,shirasaka2020phase}.\footnote{Indeed, the amplitude variable and its level sets are denoted as \textit{isostables} in the Koopman literature \cite{mauroy2013isostables}.} 
Recent advances of phase-amplitude description of oscillatory systems include higher dimensional cases \cite{perez2020global,Wilson2020}, coupled systems \cite{ermentrout2019recent,nicks2024insights,park2021high,wilson2019augmented} and piecewise-smooth oscillators \cite{coombes2023oscillatory,wilson2019isostable}.

Due to the success of the phase-amplitude variables in LC systems, it is natural to ask about the possible advantages of using an analogous set of variables to gain understanding in the important case of stochastic oscillations. 
In this paper, we will assemble recent advances in stochastic phase-amplitude functions \cite{cao2020partial,kato2021asymptotic,perez2022quantitative,perez2021isostables,schwabedal2013phase,thomas2014asymptotic}, to obtain a phase-amplitude transformation such that the stochastic dynamics are expressed ``as simply as possible". While our construction is formally analogous \textit{in the mean} to the deterministic phase-amplitude description of LC, it also applies to noise-induced oscillators: that is, to systems which despite not having an underlying LC in the absence of noise, do show stochastic oscillations because of it.

Furthermore, thanks to the mathematical simplicity of our construction, we present novel results. 
We prove the existence and uniqueness of a vector field -- that we call the \emph{effective vector field} -- that captures how the noise effectively modifies the underlying vector field to generate oscillations, even in cases where oscillations would not persist without noise. 
As our numerical examples show, the trajectories generated by the effective vector field provide an insightful interpretation of the system's dynamics.
Finally, we also derive novel expressions for the phase diffusion constant and the stationary variance of the amplitude.



Our manuscript is structured as follows.
In \S \ref{sec:section2} we briefly recall the phase-amplitude description of LC. 
In \S \ref{sec:section3} we show that, because of the fundamental differences between deterministic and stochastic dynamics, the deterministic phase and amplitude variables no longer capture the dynamics ``as simply as possible" once white Gaussian noise is introduced in the system. 
For this reason, in \S \ref{sec:section4}, we review a recently developed set of phase-amplitude functions, and show that these new phase-amplitude coordinates satisfy, under an appropriate averaging procedure, a natural simplicity criterion.
In \S \ref{sec:section5} we derive the new results in this paper, building upon the new phase-amplitude construction: namely, in Theorem \ref{thm:bigthm} we introduce the effective vector field and prove existence and uniqueness. 
In \S~\ref{sec:variances}, we derive  exact expressions for the phase diffusion constant and the stationary variance of the isostable coordinate. 
We illustrate the validity of these results by applying them to some numerical examples in \S \ref{sec:section6}. 
Finally, we conclude the paper with a discussion of the results and their possible extension in $d>2$ dimensions in \S \ref{sec:discussion}.

	\section{The Parameterization Method for Periodic Orbits}\label{sec:section2} 
    %
    Consider an autonomous system of ODEs 
	\begin{align}\label{eq:mathDef_1}
		\dot{\mbx} = \mbf(\mbx), \quad \mbx \in \mathbb{R}^{d},
	\end{align}
	whose flow is denoted by $\phi_t(\mbx)$. 
 Assume that $\mbf$ is an analytic vector field and that system \eqref{eq:mathDef_1} has a $T$-periodic hyperbolic attracting limit cycle $\Gamma$ with nontrivial Floquet exponents $\lambda_i < 0$ ($i=1,\dots,d-1$). The parameterization method guarantees the existence of the local analytical diffeomorphism \cite{cabre2005parameterization} 
	\begin{equation}\label{eq:kThetaSigma}
		\begin{aligned}
			K : \mathbb{T} \times \mathbb{R}^{d-1} &\rightarrow \mathbb{R}^d \\
			(\theta, \sigma) &\rightarrow K(\theta, \sigma), 
		\end{aligned}
	\end{equation}
	such that the dynamics of system \eqref{eq:mathDef_1} can be expressed in the simple form
	\begin{equation}\label{eq:phaseAmplitude}
		\dot{\theta} = \frac{2 \pi}{T}, \quad \quad \quad  \dot{\sigma}_i = \Lambda \sigma_i,
	\end{equation} 
	with $\Lambda = \diag(\lambda_1, \dots, \lambda_{d-1})$ and periodic boundary conditions for $\theta$. 
    Therefore, after performing the change of coordinates $\mbx = K(\theta, \sigma)$ in \eqref{eq:kThetaSigma}, the dynamics of the system \eqref{eq:mathDef_1} consists of a rigid rotation with constant velocity $2 \pi/T$ plus a contraction with exponential rates $\lambda_i$. 
    Indeed, the evolution of the flow $\phi_t(\mbx)$ in the $\theta, \sigma$ coordinates satisfies the following invariance equation
	\begin{equation}
		\phi_t(K(\theta, \sigma)) = K\left(\left[\theta + \frac{2\pi t}{T},\hspace{-3mm}\mod 2\pi\right], \sigma_i e^{\lambda_i t}\right) .
	\end{equation}

\begin{remark}\label{rm:wrapped-and-unwrapped}
Throughout the paper we will use $\theta\in\mathbb{T}\equiv[0,2\pi)$ to represent the phase on the circle, and $\vartheta\in\mathbb{R}$ to represent the \emph{unwrapped} phase on the real line, with the identification $\theta=\vartheta\mod 2\pi$.
We note that  $e^{i\theta}\equiv e^{i\vartheta}$, and  that $(\vartheta,\sigma)$ also obey the  differential equation \eqref{eq:phaseAmplitude} with natural rather than periodic boundary conditions.    
\end{remark}
 
 The map $K$ in  \eqref{eq:kThetaSigma} allows one to define a scalar function $\Psi(\mbx)$ that assigns a phase $\theta$ to any point $\mbx$ in any neighbourhood $\Omega$ of the limit cycle, contained in the basin of attraction (the stable manifold) of $\Gamma$. 
	\begin{equation}\label{eq:thetaFunc}
		\begin{aligned}
			\Psi:\Omega \subset \mathbb{R}^{d} &\to \mathbb{T},\\
			\mbx &\mapsto \Psi(\mbx) = \theta,
		\end{aligned}
	\end{equation}
	whose level curves correspond to the isochrons $\mathcal{I}_{\theta}$ 
	\begin{equation}\label{eq:isochronsDef}
		\mathcal{I}_{\theta} = \{  \mbx \in \Omega \: \mid \: \Psi(\mbx) = \theta  \} .
	\end{equation}
	
	Analogously, the map $K$ in \eqref{eq:kThetaSigma} also allows to define the scalar function $\Sigma_i$, assigning an amplitude variable $\Sigma_i$ to any point $\mbx \in$ $\Omega$: 
	\begin{equation}\label{eq:sigmaFunction}
		\begin{aligned}
			\Sigma_i : \Omega \subset \mathbb{R}^{d} &\to \mathbb{R},\\
			\mbx &\mapsto \Sigma_i(\mbx) = \Sigma_i,
		\end{aligned}
	\end{equation}
	whose level curves (see \cite{mauroy2018global, castejon2013phase}) are known as isostables $\mathcal{I}_{\sigma_i}$
	\begin{equation}\label{eq:isostablesDef}
		\mathcal{I}_{\sigma_i} = \{ \mbx \in \Omega \: \mid \: \Sigma_i(\mbx) = \sigma_i \}.
	\end{equation}
	
	Finally, note that, along trajectories, $\Psi(\mbx)$ and $\Sigma_i(\mbx)$ satisfy 
	\begin{equation}\label{eq:detFlow}
		\Psi(\phi_t(\mbx)) = \Psi(\mbx) + \frac{2 \pi t}{T} \qquad \quad \Sigma_i(\phi_t(\mbx)) = \Sigma_i(\mbx)e^{\lambda_i t}.
	\end{equation}



	
	\section{Global change of variables meets It\^{o}'s formula}\label{sec:section3}
	
	Next, we consider the following stochastic differential equation (SDE)\footnote{We now drop the assumption made in \S\ref{sec:section2} that \eqref{eq:mathDef_1} constitutes a limit cycle system.  For our governing assumptions in the stochastic case, see \S\ref{subsec:eigenfunctions-and-assumptions} and Definition \ref{def:rob-osc-crit}. }
	\begin{equation}\label{eq:SDE}
		\dot{\mbX} =\mbf(\mbX) + \mbg(\mbX)\,\xi(t), \quad \mbX \in \mathbb{R}^{d} ,
	\end{equation}
	where $\mbX = \mbX(t)$, $\mbf$ is a $d$-dimensional vector, $\mbg$ is a $d\times N$ matrix and $\xi(t)$ is white Gaussian noise $N$-dimensional vector with uncorrelated components $\langle \xi_i(t)\xi_j(t') \rangle = \delta(t-t')\delta_{ij}$. We interpret the multiplicative noise in \eqref{eq:SDE} in the sense of It\^{o} unless otherwise specified.
	
	
	Changing from deterministic to stochastic dynamics results in many fundamental changes: orbits are no longer periodic and the transit times between isochrons or other Poincar\'e sections are random variables \cite{schwabedal2013phase}. Moreover, the classical notion of a ``limit cycle", as a closed, isolated periodic orbit, is no longer well defined. Additionally, the change of variables does not follow standard practice but the It\^{o} rule \cite{Gar85}. Indeed, for any smooth function $\Phi(\mbx)$, It\^{o}'s lemma gives the increment in $\Phi$ as
	\begin{equation}\label{eq:itoRule} 
	\begin{aligned}
		\dot{\Phi}(\mbX)&=
		\sum_{j=1}^d \mbf_j(\mbX)\frac{\partial\Phi(\mbX)}{\partial \mbx_j}
		+
		\sum_{j,k=1}^d\mathcal{G}_{jk}(\mbX)\frac{\partial^2\Phi(\mbX)}{\partial \mbx_j \partial \mbx_k} +\sum_{j=1}^d\sum_{k=1}^N\frac{\partial\Phi(\mbX)}{\partial \mbx_j}\mbg_{jk}(\mbX)\xi_{k}(t),\\
		 &= \mathcal{L}^\dagger [\Phi(\mbX)] + \sum_{j=1}^d\sum_{k=1}^N\frac{\partial\Phi(\mbX)}{\partial \mbx_j}\mbg_{jk}(\mbX)\xi_{k}(t),
		\end{aligned}
	\end{equation} 
	with $ \mathcal{G}(\mbx)=\frac12 \mbg(\mbx)\mbg(\mbx)^\intercal$ and
    \begin{equation}\label{eq:ldagger}
    \mathcal{L}^\dagger = \mbf(\mbx)\cdot\nabla_\mbx+\mathcal{G}_{jk}(\mbx)\partial^2_{\mbx_j \mbx_k}\enskip,
    \end{equation}
    the Kolmogorov backwards operator.  We adopt the Einstein summation convention for repeated indices.

	The success of the  phase-amplitude transformation of LC systems motivates the search for a similar change of coordinates facilitating the analysis of the generic SDE \eqref{eq:SDE}. 
    However, because of the extra second derivative term in \eqref{eq:itoRule}, the transformation to the deterministic phase amplitude variables $\Psi(\mbx), \Sigma_i(\mbx)$ no longer has a simple expression.
    Indeed, along trajectories of the process \eqref{eq:SDE}, we have:
	\begin{equation}
		\begin{aligned}\label{eq:stocPhaseAmpVF}
			\dot{\Psi}(\mbX)=&~  \frac{2 \pi}{T} + \sum_{j,k=1}^d\mathcal{G}_{jk} (\mbX)\frac{\partial^2\Psi (\mbX)}{\partial \mbx_j\partial \mbx_k} +\sum_{j=1}^d\sum_{k=1}^N \mbg_{jk}(\mbX) \frac{\partial\Psi (\mbX)}{\partial \mbx_j}\,\xi_{k}(t)\enskip, \\
			\dot{\Sigma}_i(\mbX)=&~  \lambda_i\Sigma_i(\mbX) +  \sum_{j,k=1}^d\mathcal{G}_{jk}(\mbX) \frac{\partial^2\Sigma_i(\mbX)}{\partial \mbx_j\partial \mbx_k}+\sum_{j=1}^d\sum_{k=1}^N \mbg_{jk}(\mbX) \frac{\partial\Sigma_i(\mbX)}{\partial \mbx_j}\,\xi_{k}(t)\enskip,
		\end{aligned}
	\end{equation}	
	where $T$ is the period of the deterministic limit cycle \eqref{eq:mathDef_1}, and we used $\mbf(\mbx)\cdot \nabla\Psi(\mbx)=2 \pi /T$ and $\mbf(\mbx)\cdot\nabla\Sigma_i(\mbx)=\lambda_i \Sigma_i(\mbx)$ for all $\mbx = K(\theta, \sigma)$.
	
	\section{A Phase-Amplitude description of Stochastic Oscillators}\label{sec:section4}
	
	As Eq.~\eqref{eq:stocPhaseAmpVF} shows, for the SDE \eqref{eq:SDE}, neither the average local rate of change of the phase variable nor the isostable coordinate $\sigma_i$ satisfy the conditions $\langle d\theta\rangle=\frac{dt}{T}$ and $\langle d\sigma_i\rangle=\lambda_i \sigma_i\,dt$ that define them, not even ``on average". 
    In addition, as they rely on the deterministic phase-amplitude functions, equations in system \eqref{eq:stocPhaseAmpVF} are \textit{a priori} restricted to systems having an underlying limit cycle.
    Thus they are not applicable to the important and broad family of noise-induced oscillators, i.e.~systems in which the noise is fundamental for generating oscillations, such as quasicycle systems \cite{lugo2008quasicycles, bressloff2010metastable, wallace2011emergent, thomas2019phase},  excitable systems \cite{lindner2004effects, schwabedal2010effective, gutkin2022theta, zhu2022phase}, or noisy heteroclinic cycle systems \cite{coller1994control, stone1999noise, armbruster2003noisy, bakhtin2011noisy, horchler2015designing, ashwin2016quantifying, giner2017power, rouse2021visualization, barendregt2023heteroclinic, jeong2023effect}  . 

    Therefore, the expressions in system~\eqref{eq:stocPhaseAmpVF} could be interpreted as necessitating a redefinition of the concepts of ``asymptotic phase" and  ``isostables''. 
    Following the main idea of the Parameterization Method, we will look for the variables for which the dynamics are expressed, \emph{in the mean}, as simply as possible. 
    
    We now turn to the special case of planar\footnote{See \S\ref{sec:discussion} for discussion of possible extensions of the framework presented here beyond the planar case.} systems ($d=2$). 
    Let us consider the previous equation \eqref{eq:itoRule}
    \begin{equation}
		\dot{\Phi}(\mbX)= \mathcal{L}^\dagger [\Phi(\mbX)] + \sum_{j=1}^2\sum_{k=1}^N\frac{\partial\Phi(\mbX)}{\partial \mbx_j}\mbg_{jk}(\mbX)\xi_{k}(t),
	\end{equation} 
    and let us now average
    \begin{equation}\label{eq:dynkin-inf}
		\frac{d}{dt}{\E^{\mbx_0}[\Phi}(\mbX)]= \E^{\mbx_0}[\mathcal{L}^\dagger [\Phi(\mbX)]],
	\end{equation}
    where $\E^{\mbx_0}$ indicates the expectation value over an ensemble of trajectories starting from $\mbX(0) = \mbx_0$.
    This expression is the differential form of Dynkin's formula.\footnote{Dynkin's theorem (Øksendal 2007, \S7.4) may be expressed in integral form as follows:
    $E^{\mbx}\left[\phi(\mbX(\tau))\right]=\phi(\mbx)+E^{\mbx}\left[\int_0^\tau\mathcal{L}^\dagger\left[\phi(\mbX(s))\right]\,ds\right]$, where $\tau$ is any stopping time for which $E^\mbx[\tau]<\infty$, $\phi$ is any smooth function, and $E^\mbx$ denotes expectation with respect to the law of the process with initial condition $\mbx$.  
    }
    If we were now to find functions $\Theta(\mbx)$ and  $\BSigma(\mbx)$ satisfying 
    \begin{equation}\label{eq:ldaggerEq}
	    	\mathcal{L}^\dagger [\Theta(\mbx)] = \frac{2 \pi}{\tbar}, \qquad \quad \mathcal{L}^\dagger [\BSigma(\mbx)] = \lambda_\text{Floq} \BSigma(\mbx),
	\end{equation}
	with $\tbar, \lambda_\text{Floq} \in \mathbb{R}$, then using \eqref{eq:dynkin-inf} we would obtain
    \begin{equation}\label{eq:eq:stocPhaseAmpVF}
			\frac{d}{dt}\E^{\mbx_0}[\Theta(\mbX)]=~  \frac{2 \pi}{\tbar}, \qquad   
   \frac{d}{dt}\E^{\mbx_0}[\BSigma(\mbX)]= \lambda_\text{Floq} \E^{\mbx_0}[\BSigma(\mbX)].
    \end{equation}
    In these coordinates, the trajectory would be given simply as
    \begin{equation}\label{eq:phase-amp-avg}
    \E^{\mbx_0}[\Theta(\mbX)]=~  \frac{2 \pi t}{\tbar} + \Theta(\mbx_0), \qquad   
   \E^{\mbx_0}[\BSigma(\mbX)]= \BSigma(\mbx_0)e^{\lambda_\text{Floq}t}.
    \end{equation}
    Thus we propose these functions, $\Theta$ and $\BSigma$, as our stochastic phase-amplitude variables, since we see that upon averaging they satisfy expressions that formally resemble the deterministic PM version (see \eqref{eq:detFlow}).
    
\begin{remark}\label{rm:unwrapped-phase}
    When treating averages of phase variables, such as $\mathbb{E}^{\mbx_0}[\Theta(\mbX)]$ in Equation \eqref{eq:phase-amp-avg}, we have in mind the unwrapped phase $\vartheta\in\mathbb{R}$ rather than the wrapped phase  
    $\theta~\mod 2\pi\in\mathbb{T}$ (see Remark~\ref{rm:wrapped-and-unwrapped}).
\end{remark} 

    \subsection{Stochastic oscillators described by eigenfunctions}
    \label{subsec:eigenfunctions-and-assumptions}

    Finding the functions $\Theta(\mbx)$ and  $\BSigma(\mbx)$ (and their associated parameters $\tbar$ and $\lambda_\text{Floq}$) will require us to review recent advances in the field of stochastic oscillations \cite{perez2022quantitative,perez2023universal,perez2021isostables,thomas2014asymptotic}.  
    The key conceptual observation underlying these recent advances requires to approach stochastic oscillations by means of an ensemble of trajectories instead of focusing on individual trajectories. Therefore, we consider the conditional density
    \begin{equation}
		\rho(\mbx,t\given\mbx_0,s)=\frac{1}{|d\mbx|}\Pr\left\{\mbX(t)\in[\mbx,\mbx+d\mbx)\given \mbX(s)=\mbx_0 \right\}, \qquad \text{for} \enskip t>s,
	\end{equation} 
    whose evolution follows the Fokker-Planck equation
    \begin{equation}
    \label{eq:forward-eq}
		\frac{\partial}{\partial t}P(\mbx,t\given\mbx_0,s)=\LL[P]
        =-\nabla_\mbx\!\cdot\!\left( \mbf(\mbx) P \right) + \sum_{i,j} \frac{\partial^2}{\partial x_i x_j}\left(\mathcal{G}_{ij}(\mbx)P\right),
    \end{equation}
    where $\mathcal{G}=\frac12 \mbg\mbg^\intercal$ (see \eqref{eq:itoRule}). 
The formal adjoint of the operator $\LL$ is Kolmogorov's \emph{backward} operator $\LLd$ (also known as the generator of the Markov process \eqref{eq:SDE} \cite{vcrnjaric2019koopman,kato2021asymptotic}, and previously defined in \eqref{eq:ldagger} in the context of the It\^{o} rule), which satisfies the equation
\begin{equation}
  \label{eq:backward-eq}
		-\frac{\partial}{\partial s}P(\mbx,t\given\mbx_0,s)=\LLd[P]
  =\mbf(\mbx_0) \cdot \nabla_{\mbx_0}\!\left(P \right)+\sum_{i,j}\mathcal{G}_{ij}(\mbx_0)\frac{\partial^2 P}{\partial x_{0,i} x_{0,j}}.
\end{equation}
We will assume that the operators $\LL$, $\LLd$ possess a discrete set of  eigenfunctions
	\begin{equation}\label{eq:spectral-eigenfun}
		\LL[P_\lambda]=\lambda P_\lambda,\qquad\LLd[Q^*_{\lambda}]=\lambda Q^*_{\lambda},
	\end{equation}
    which satisfy in turn the following orthogonality relationship
    \begin{equation}\label{eq:spectral-decomposition}
        \langle Q_{\lambda'}\given P_{\lambda}\rangle= \int d\mbx\, Q^*_{\lambda'}(\mbx)P_\lambda(\mbx)  =  \delta_{\lambda'\lambda}.
    \end{equation}
    Moreover, we assume the forward equation \eqref{eq:forward-eq} satisfies a unique steady-state probability distribution $P_0(\mbx)$. 
    In place of the basin of attraction $\Omega$ considered in the deterministic limit-cycle context, here we take the domain of our analysis to be the support of $P_0$, that is $\mathcal{D}=\text{supp}(P_0)\equiv\{\mbx\:\mid\:P_0(\mbx)>0\}$.  
    
    Under these conditions, the transition probability can be expressed as (see \cite{Gar85}) 
	\begin{equation}\label{eq:condDensity}
		P(\mbx,t|\mbx_0,s) = P_0(\mbx) + \sum_{\lambda\not=0} e^{\lambda(t-s)} P_\lambda(\mbx) Q^*_\lambda(\mbx_0),
	\end{equation}
    for $t>s$.
    That is, the transition probability $P$ can be regarded as a sum of modes, each of which decays at a rate given by the real part of its respective eigenvalue $\lambda$, leading in the long-time limit to the stationary distribution $P_0(\mbx)$.\\
    
    As pointed out by Thomas and Lindner in \cite{thomas2014asymptotic}, the decaying modes in \eqref{eq:condDensity} store important information about the stochastic oscillator provided the following set of conditions (which they coined as ``robustly oscillatory'') are met
    	\begin{definition}[Criteria for a robust stochastic oscillator]\label{def:rob-osc-crit}
	Assume the spectra of the operator $\LLd$ satisfies the following conditions
	\begin{itemize}
	    \item its nontrivial eigenvalue with least negative real part $\lambda_1 = \mu + i\omega$ is complex, and unique (occurs with algebraic multiplicity one).  
	    \item all other nontrivial eigenvalues $\lambda'$ be significantly more negative, that is, $\Re[\lambda'] < 2\mu$.
	    \item $|\omega/\mu| \gg 1$
	\end{itemize}
	
	\end{definition}
    
    The first ``robustly oscillatory" condition, requires that the nontrivial eigenvalue in \eqref{eq:spectral-decomposition} with least negative real part $\lambda_1 = \mu \pm i\omega$ is part of a complex-conjugate eigenvalue pair, each of which is unique (occurs with algebraic multiplicity one).  
	If this condition is met, then we can express its associated right (forward) and left (backward or adjoint) eigenfunctions in polar form as $P_{\lambda_1}(\textbf{x}) = v(\mbx)e^{-i\phi(\mbx)}$ and 
	$Q^*_{\lambda_1}(\textbf{x}_0) = u(\mbx_0)e^{i\psi(\mbx_0)}$, where $v(\mbx)\ge 0$ and $u(\mbx_0)\ge 0$ are real functions specifying the amplitude of the corresponding eigenfunction.\footnote{According to \eqref{eq:spectral-eigenfun}, the forward and backward eigenfunctions should be written as $P_{\lambda_1}$ and $Q^*_{\lambda_1}$, respectively. However, in what follows we will denote them as $P_{1}$ and $Q^*_{1}$ to ease notation.} 
	
	The second condition requires all other nontrivial eigenvalues $\lambda'$ to be significantly more negative, that is, $\Re[\lambda'] < 2\mu$.
		If this is the case, then, at sufficiently long times the sum in \eqref{eq:condDensity} can be written  as
		\begin{equation*}\label{eq:decayingModes}
			\frac{\rho(\textbf{x},t|\textbf{x}_0,s) - P_0(\textbf{x})}{2 v(\mbx)u(\mbx_0)} \approx e^{\mu(t-s)}\cos(\omega(t-s) + \psi(\mbx_0) - \phi(\mbx)).
		\end{equation*}
		This asymptotic form means that the density approaches its steady state as a damped focus, with an oscillation period  of $2\pi/\omega$, and a decaying amplitude with time constant $1/|\mu|$. 
	
	The third condition focuses on the ``quality factor" $|\omega/\mu|$ \cite{giner2017power}. 
 The condition $|\omega/\mu| \gg 1$ is required to ensure the oscillation completes sufficiently many rotations before the damping reduces its phase coherence beyond detectability.

    \subsection{The stochastic asymptotic phase and its connection with the MRT phase} 

    The robustly oscillatory criterion provides a set of conditions guaranteeing that the asymptotic behaviour of the density \eqref{eq:decayingModes} has a significant oscillating component as it decays towards the steady-state distribution. 
    Under these conditions, Thomas and Lindner showed in \cite{thomas2014asymptotic} that one can extract a meaningful phase function $\psi(\mbx)$ --which they called the ``stochastic asymptotic phase"-- by  taking the complex argument of the slowest complex decaying backwards mode $Q^*_{1}(\mbx)=u(\mbx)e^{ i\psi(\mbx)}$. That is,
	\begin{equation}\label{eq:tlPhase}
		\psi(\mbx) =  \arg(Q^*_{1}(\mbx)),
	\end{equation}
	provided $u(\mbx) \neq 0$. However, as we showed in \cite{perez2022quantitative}, 
	\begin{equation}\label{eq:tlExpVal}
		\LLd[\psi(\mbx)] = \omega - 2\sum_{i,j} \mathcal{G}_{ij}(\mbx) \partial_i \ln(u(\mbx)) \partial_j \psi(\mbx).
	\end{equation}
    Therefore, the stochastic asymptotic phase $\psi(\mbx)$ does not fulfill the condition stated in \eqref{eq:ldaggerEq}. Indeed, the function $\Theta(\mbx)$ in \eqref{eq:ldaggerEq} that we are after, is equivalent to finding a function whose mean--return-time across trajectories is constant. In other words, the level sets of $\Theta(\mbx)$ will behave \textit{in the mean} as the deterministic isochrons $\mathcal{I}_\theta$ in \eqref{eq:isochronsDef}.

    Schwabedal and Pikovsky proposed in \cite{schwabedal2013phase} to construct a function $\Theta(\mbx)$ with uniform mean--return-time properties by way of an iterative numerical procedure.  They introduced a system of Poincaré sections of the form $\{ \ell_{MRT}(\phi), \enskip 0 \leq \phi \leq 2\pi \}$ which foliate the basin of attraction of the limit cycle $\mathcal{R} \subset \mathbb{R}^2$.  
    The sections $\ell_\text{MRT}$ exhibit the mean--return-time property in the sense that $\forall \mbx \in \ell_\text{MRT}$ the mean return time from $\mbx$ back to $\ell_\text{MRT}$ (having first completed a full oscillation) is constant. 

    Following Schwabedal and Pikovsky's seminal work, Cao and colleagues showed that the MRT phase could be obtained efficiently as the solution of a boundary value problem  \cite{cao2020partial}. 
    The $\ell_{MRT}$ sections correspond to the level curves of a function $T(\mbx)$ that satisfies a partial differential equation
    \begin{equation}\label{eq:meanIsochrons}
		\LLd[T(\mbx)] = -1,
	\end{equation}
    with specific boundary conditions.
    Here $\LLd$ corresponds, again, to the ubiquitous Kolmogorov backwards operator in \eqref{eq:itoRule}.
	
	Upon imposing a boundary condition reflecting a jump by $\tbar$ (the mean return period of the oscillator)\footnote{The  sufficient conditions to guarantee a well-defined mean period $\tbar$ are discussed in \cite{cao2020partial}.} across an arbitrary section transverse to the oscillation, the \textit{unique} solution of \eqref{eq:meanIsochrons}, up to an additive constant $T_0$, is a version of the MRT function,
	\begin{equation}
		\Theta(\mbx)=(2\pi/ \tbar)(T_0-T(\mbx)).    
	\end{equation}
	Hence, as shown in \cite{cao2020partial}, the MRT phase $\Theta(\mbx)$ satisfies the following equation
	\begin{equation}\label{eq:MRTphase}
		\mathcal{L}^\dagger [\Theta(\mbx)] = \frac{2 \pi}{\tbar},
	\end{equation}  
    which does match the condition \eqref{eq:ldaggerEq}. 
    We refer the reader to  \cite{cao2020partial, perez2022quantitative} for further details on computing $\tbar$, and the relationship between $\psi(\mbx)$ and $\Theta(\mbx)$.

    \subsection{The stochastic amplitude}


    In order to find the stochastic analogue of the isostable function, it is important to point out that, unlike the MRT function, the isostable function corresponds with an eigenfunction of $\LLd$ (see \eqref{eq:ldaggerEq}). 
    Under the assumptions of the ``robustly oscillator" criterion, we showed in \cite{perez2021isostables} that the stochastic isostable function $\BSigma(\mbx)$ is given by the slowest decaying eigenmode describing pure contraction without an associated oscillation. 
    That is,  the stochastic isostable function $\BSigma(\mbx)$ shall correspond with the eigenfunction of $\LLd$ associated to the real, \textit{least negative} eigenvalue, which we denote as $\lambda_\text{Floq} <0$. 
    
    Therefore, since $\mathcal{L}^\dagger [\BSigma(\mbx)] = \lambda_\text{Floq} \BSigma(\mbx)$, following the It\^{o} rule (Eq.~\eqref{eq:itoRule}), it is straightforward to show that  the stochastic isostable function follows
    \begin{equation}
    \dot{\BSigma}(\mbX)=~  \lambda_{\text{Floq}}\BSigma(\mbX) + \sum_{j=1}^2\sum_{k=1}^N \mbg_{jk}(\mbX) \frac{\partial\BSigma(\mbX)}{\partial \mbx_j}\,\xi_{k}(t).
    \end{equation}
    That is, transforming the dynamics $\mbX(t)$ in \eqref{eq:SDE} to $\BSigma(\mbX(t))$ yields a one-dimensional noise process (with multiplicative noise), consistent with the idea of finding a change of variables giving the ``simplest possible" dynamical description.
    Consequently, when transformed to the amplitude variables, the trajectories fluctuate around the set
    \begin{equation}
        \Sigma_0 = \{\mbx \mid \BSigma(\mbx)=0\}.
    \end{equation}
    Thus the zero level set of the slowest decaying real mode of the backward Kolmogorov operator is a natural candidate for the term ``stochastic limit cycle".  

\section{New insights into the stochastic dynamics provided by the transformation to stochastic phase-amplitude dynamics}\label{sec:section5} 
So far, we have reviewed previous results to assemble the following phase-amplitude description of the Langevin equation in \eqref{eq:SDE}
\begin{equation}
\begin{aligned}
\label{eq:stocPhaseAmp}
			\dot{\Theta}(\mbX)=&~  \frac{2 \pi}{\tbar} +\sum_{j=1}^2\sum_{k=1}^N \mbg_{jk}(\mbX) \frac{\partial\Theta (\mbX)}{\partial \mbx_j}\,\xi_{k}(t), \\
			\dot{\BSigma}(\mbX)=&~  \lambda_{\text{Floq}}\BSigma(\mbX) + \sum_{j=1}^2\sum_{k=1}^N \mbg_{jk}(\mbX) \frac{\partial\BSigma(\mbX)}{\partial \mbx_j}\,\xi_{k}(t).
\end{aligned}
\end{equation}
While the MRT phase has been considered previously \cite{schwabedal2013phase,cao2020partial} and also the stochastic isostable function \cite{perez2021isostables}, these functions have not previously been combined to give a complete change of coordinates in the spirit of the Parameterization Method \cite{guillamon2009computational, perez2020global}. 

In the remaining pages we detail our new results for planar stochastic oscillators.  
We prove results on the effective vector field, and we derive the second order moments (and therefore the variances) of $\Theta(\mbx)$ and $\BSigma(\mbx)$.


\subsection{The effective vector field}

As we discussed in Section~\S\ref{sec:section2}, for deterministic systems $\dot{\mbx} = \mbf(\mbx)$, with an underlying LC, there is a general procedure to find its associated phase and amplitude functions.\footnote{See \cite{perez2020global,mauroy2018global} for PM and Koopman-theoretic ways of finding such deterministic phase-amplitude functions, respectively.} 
By contrast, the previously introduced phase-amplitude functions $\Theta(\mbx)$ and $\BSigma(\mbx)$ do not follow from a deterministic vector field but from the backwards operator $\LLd$. 
Therefore, it seems natural to seek the deterministic vector field having these specific phase and amplitude functions. 
This question motivates our next result, in which we prove the existence and uniqueness of such a vector field (which we denote  the \textit{effective vector field}, see Remark~\ref{rm:eff-vec-field}).

\begin{theorem}\label{thm:bigthm}
    Assume there exist $\Theta(\mbx)$ and $\BSigma(\mbx)$, which are $C^1$ on an open subset $\mathcal{D}'\subset\mathcal{D}=\text{supp}(P_0)$, satisfying
        \begin{equation}\label{eq:ldaggerEq-in-theorem}
	    	\mathcal{L}^\dagger [\Theta(\mbx)] = \frac{2 \pi}{\tbar}, \qquad \quad \mathcal{L}^\dagger [\BSigma(\mbx)] = \lambda_\text{Floq} \BSigma(\mbx),
	\end{equation}
 where $\tbar$ is the mean oscillator period.
 Moreover, let the gradients $\nabla\Theta$ and $\nabla\BSigma$ be linearly independent on $\mathcal{D}'$.
    Then, there exists a vector field $\mbF$ on $\mathcal{D}'$, with associated flow $\varphi_t(\mbx)$, such that
    \begin{equation*}
       \E[\Theta(\mbX(t))] = \Theta(\varphi_t(\mbx_0)), \qquad \E[\BSigma(\mbX(t))] = \BSigma(\varphi_t(\mbx_0)),
    \end{equation*}
    for any $\mbx_0 = \mbX(0)\in\mathcal{D}'$.    
    Moreover, the vector field $\mbF$ is unique.
\end{theorem}

\begin{proof}
  	To prove this result we exploit the fact that the mean dynamics for $\Theta(\mbx)$ and $\BSigma(\mbx)$ are \textit{deterministic}. Therefore, we can ask for the vector field 
	\begin{equation}
		\dot{\mbx} = \mbF(\mbx)
	\end{equation}
	whose flow we denote as $\varphi_t(\mbx_0)$, generating such mean phase-amplitude dynamics,
 in the sense that $\varphi_0(\mbx_0)=\mbx_0$ and $d\varphi_t(\mbx_0)/dt=\mbF(\varphi_t(\mbx_0))$. 
    That is, we seek the vector field for which $\varphi_t$ satisfies
	\begin{equation}\label{eq:stocFlow}
		\Theta(\varphi_t(\mbx_0)) = \Theta(\mbx_0) + \frac{2 \pi t}{\tbar}, \qquad \BSigma(\varphi_t(\mbx_0)) = \BSigma(\mbx_0)e^{\lambda_\text{Floq}t}.
	\end{equation}
	
	Taking time derivatives in \eqref{eq:stocFlow}, and applying the chain rule, we see that $\mbF$ must satisfy
	\begin{equation}
		\label{eq:effective-vector-field-condition}
		\nabla \Theta(\textbf{x}) \cdot \mbF(\textbf{x}) = \frac{2 \pi}{\tbar}, \qquad   \nabla \BSigma(\textbf{x}) \cdot \mbF(\textbf{x}) = \lambda_\text{Floq} \BSigma(\textbf{x}).
	\end{equation}
    Because the gradients of $\BSigma(\textbf{x})$ and $\Theta(\textbf{x})$ are linearly independent at each point $\mbx$ in the domain $\mathcal{D}'$, by assumption,
	we obtain $\mbF(\textbf{x})$ by  matrix inversion:
    \begin{equation}
        \mbF(\mbx) = 
        \left(\begin{array}{c}
        \nabla \Theta(\textbf{x})^\intercal\\
        \nabla \BSigma(\textbf{x})^\intercal
        \end{array}\right)^{-1}
        \left(\begin{array}{c} 2 \pi/\tbar \\ \lambda_\text{Floq} \BSigma(\textbf{x})\end{array}\right).
    \end{equation}
    Linear independence guarantees that the matrix is invertible and the solution is unique.
\end{proof}

\begin{remark}\label{rm:eff-vec-field}
    As the numerical examples in Section \S\ref{sec:section6} will show, as long as the conditions of the ``robustly oscillatory criteria'' in \ref{def:rob-osc-crit} are met,  the deterministic
vector field $\mbF$ in \eqref{eq:effective-vector-field-condition} will generate, by construction, a flow with a stable LC given by the set $\Sigma_0 = \{\mbx \mid \BSigma(\mbx)=0\}$ whose nontrivial Floquet exponent will be given by $\lambda = \lambda_\text{Floq}$. 
Moreover, this limit cycle will have period $\tbar$ and its isochrons will correspond with the level curves of $\Theta(\mbx)$. For these reason we denote $\mbF$ as the \textit{effective} vector field because it captures  how the noise \textit{effectively} modifies the underlying vector field $\mbf$ of the SDE in \eqref{eq:SDE} to sustain oscillations. 
We emphasize that this construction remains valid even if the deterministic system
 $\dot{\mbx} = \mbf(\mbx)$ does not have oscillatory dynamics.
\end{remark}


\subsection{Calculation of variances}\label{sec:variances} 
In contrast to limit cycle systems in ordinary differential equations, the stochastic phase and amplitude variables in \eqref{eq:stocPhaseAmp} do fluctuate. 
So far, we have been able to find the equations giving the evolution of the first moment (the mean) of the phase-amplitude functions $\Theta(\mbx)$ and $\BSigma(\mbx)$ (see \eqref{eq:eq:stocPhaseAmpVF}).
However, the more we know about the evolution of higher moments, the better we will understand the behavior of the random variables of interest.
Therefore, in this Section we derive expressions for the variances of the proposed phase-amplitude variables. To start, consider the square of an arbitrary $C^2$ function  $\Upsilon(\mbx)$, and let us compute 
\begin{equation}\label{eq:lDagger-squared}
\begin{aligned}
    \LLd[\Upsilon^2(\mbx)] &= \mbf(\mbx) \cdot \nabla_{\mbx}\!\left(\Upsilon^2(\mbx) \right)+\sum_{i,j}\mathcal{G}_{ij}(\mbx)\frac{\partial^2 \Upsilon^2(\mbx)}{\partial x_{i} x_{j}}\\
    &= 2 \Upsilon(\mbx) \Big(\mbf(\mbx) \cdot \nabla_{\mbx}\!\left(\Upsilon(\mbx) \right)\!\Big)+2\sum_{i,j}\mathcal{G}_{ij}(\mbx)\frac{\partial }{\partial x_{i}}\left(\Upsilon(\mbx) \frac{\partial \Upsilon(\mbx) }{\partial x_{j}} \right)\\
    &= 2 \Upsilon(\mbx) \Big(\mbf(\mbx) \cdot \nabla_{\mbx}\!\left(\Upsilon(\mbx) \right)\!\Big)+2\sum_{i,j}\mathcal{G}_{ij}(\mbx)\left(\frac{\partial \Upsilon(\mbx)}{\partial x_{i}}\frac{\partial \Upsilon(\mbx) }{\partial x_{j}} + \Upsilon(\mbx)\frac{\partial^2 \Upsilon(\mbx)}{\partial x_{i} x_{j}} \right)\\
    &= 2\Upsilon(\mbx)\LLd[\Upsilon(\mbx)] +  2\sum_{i,j}\mathcal{G}_{ij}(\mbx)\left(\frac{\partial \Upsilon(\mbx)}{\partial x_{i}}\frac{\partial \Upsilon(\mbx) }{\partial x_{j}} \right).
\end{aligned}
\end{equation}

Therefore, transforming the SDE in \eqref{eq:SDE} to $\Upsilon^2(\mbX)$, using the It\^{o} rule and $\LLd[\Upsilon^2(\mbx)]$ in \eqref{eq:lDagger-squared} we find   
\begin{equation}\label{eq:SDE_VAR}
		\frac{d}{dt}\Upsilon^2(\mbX) = 2\Upsilon(\mbX)\LLd[\Upsilon(\mbX)] +  2\sum_{i,j}\mathcal{G}_{ij}(\mbX)\left(\frac{\partial \Upsilon(\mbX)}{\partial x_{i}}\frac{\partial \Upsilon(\mbX) }{\partial x_{j}} \right) + \mbg(\mbX)\nabla \Upsilon^2(\mbX)\xi(t),
\end{equation}
which we can average to obtain an equation for the time evolution of the second moment of $\mathbf{\Upsilon}$.
Noting the independence of $\xi(t)$ and $\mbg(\mbX(t))\nabla \Upsilon^2(\mbX(t))$,  we obtain\footnote{For economy of notation, in  \S\ref{sec:variances} we use $\lr{\cdot}_{\mbx_0}$ instead of $\E^{\mbx_0}[\cdot]$ to denote the expectation with respect to the law of the process with initial condition $\mbx_0$.}:
\begin{equation}\label{eq:square-avg}
    \frac{d}{dt}\lr{\Upsilon^2(\mbX)}_{\mbx_0} = 2 \lr{\Upsilon(\mbX)\LLd[\Upsilon(\mbX)]}_{\mbx_0} +  2\sum_{i,j} \lr{ \mathcal{G}_{ij}(\mbX)\left(\frac{\partial \Upsilon(\mbX)}{\partial x_{i}}\frac{\partial \Upsilon(\mbX) }{\partial x_{j}} \right) }_{\mbx_0} .  
\end{equation}

\subsubsection{The MRT case}\label{sec:mrt-variance}

We now wish to use the previous calculations to compute the variance of the unwrapped phase $\Theta(\mbx)$, cf.~Remark \ref{rm:unwrapped-phase}. 
To that aim, we consider the preceding equation \eqref{eq:square-avg}  substituting $\Upsilon^2(\mbX)$ by  $\Theta^2(\mbX)$:
\begin{equation}\label{eq:phase-var-expectation}
\begin{aligned}
    \frac{d}{dt}\!\lr{\Theta^2(\mbX)}_{\mbx_0} \!&= 2 \lr{\Theta(\mbX)\LLd[\Theta(\mbX)]}_{\mbx_0}  +  2\sum_{i,j} \lr{ \mathcal{G}_{ij}(\mbX)\left(\frac{\partial \Theta(\mbX)}{\partial x_{i}}\frac{\partial \Theta(\mbX) }{\partial x_{j}} \right) }_{\mbx_0}. 
\end{aligned}
\end{equation}

\begin{remark}\label{rm:phase-cte-div} 
    A well-defined phase function for a $d$-dimensional dynamical system provides a mapping from $\mathbb{R}^d$ to the circle.
    From topological considerations \cite{winfree1980geometry}, we  expect the domain of this mapping to have at least one ``hole" at which the phase is not defined. 
    At such points, which we collect in the \emph{phaseless set} $\mathcal{U}$ \cite{winfree1974patterns,guckenheimer1975}, the phase function is not $C^2$ (and may not be defined at all).
    (In Theorem \ref{thm:bigthm} we exclude such points, since $\mathcal{U}\cap\mathcal{D}'=\emptyset$.)
    Similarly, as one approaches $\mathcal{U}$, we expect the derivatives $\partial_{\mbx}\Theta(\mbx)$ to diverge. 
    Hence, computing the expectation in \eqref{eq:phase-var-expectation}
    requires careful treatment of the phaseless points $\mbx \in \mathcal{U}$. 
To this end, for $\epsilon>0$ we define $\mathcal{U}_\epsilon=\{\mbx\in\mathcal{D}'\given \rho(\mbx,\mathcal{U})<\epsilon\},$
where $\rho(\mbx,\mby)$ is the Euclidean distance on $\mathbb{R}^d$.
That is, $\mathcal{U}_\epsilon=\mathcal{U}+\mathcal{B}_\epsilon(0)$, the open set that extends out from $\mathcal{U}$ by $\epsilon$ in every direction.  
The diffusion constant $D^\Theta_{\text{eff}}$ thus may depend on $\epsilon$. 
Whether this dependence is weak or strong we examine case-by-case below.  
In  Section $\S\ref{sec:section6}$ we compute the diffusion constant for three different examples. For each of them, we discuss the proper numerical treatment leading to a meaningful computation of formula \eqref{eq:thetaVar} and hence of a proper estimation of the diffusion constant  $D^\Theta_{\text{eff}}$. 
\end{remark}

Taking into account the previous Remark~\ref{rm:phase-cte-div}, we compute the expectation in \eqref{eq:phase-var-expectation} as follows.
\begin{equation*} \small
\begin{aligned}  
    \frac{d}{dt}\!\lr{\Theta^2(\mbX)}_{\mbx_0} \!    &= 2 \lr{\Theta(\mbX)\left( \frac{2\pi}{\tbar} \right)}_{\mbx_0}  +  2\sum_{i,j} \int_{\mbx \notin \mathcal{U}_\epsilon} d\mbx ~P(\mbx, t| \mbx_0, 0)  ~\mathcal{G}_{ij}(\mbx)\left(\frac{\partial \Theta(\mbx)}{\partial x_{i}}\frac{\partial \Theta(\mbx) }{\partial x_{j}} \right)  \\
    &= \frac{4\pi}{\tbar} \lr{\Theta(\mbX)}_{\mbx_0} +  2\sum_{i,j} \!\int_{\mbx \notin \mathcal{U}_\epsilon}\! d\mbx \left(\! P_0(\mbx) + \sum_{\lambda\not=0} Q^*_\lambda(\mbx_0)P_\lambda(\mbx) e^{\lambda t}\!\right)\! \mathcal{G}_{ij}(\mbx)\! \left(\frac{\partial \Theta(\mbx)}{\partial x_{i}}\frac{\partial \Theta(\mbx) }{\partial x_{j}} \right).
\end{aligned}
\normalsize
\end{equation*}\normalsize
Integrating this expression gives
\begin{equation*}
\begin{aligned}
\lr{\Theta^2(\mbX)}_{\mbx_0} - \Theta^2_0    &= \frac{4 \pi}{\tbar} \int^t_0  \! \left(\frac{2 \pi t'}{\tbar} + \Theta_0 \right)dt' + 2\sum_{i,j} \int_{\mbx \notin \mathcal{U}_\epsilon}  \!\! P_0(\mbx)  \mathcal{G}_{ij}(\mbx)\! \left(\frac{\partial \Theta(\mbx)}{\partial x_{i}}\frac{\partial \Theta(\mbx) }{\partial x_{j}} \right) d\mbx \! \int^t_0 \!dt' \\ &+ 2\sum_{\lambda\not=0} Q^*_\lambda(\mbx_0) \sum_{i,j} \int_{\mbx \notin \mathcal{U}_\epsilon}   P_\lambda(\mbx)  \mathcal{G}_{ij}(\mbx)\left(\frac{\partial \Theta(\mbx)}{\partial x_{i}}\frac{\partial \Theta(\mbx) }{\partial x_{j}} \right) d\mbx \int^t_0 e^{\lambda t} dt'\\
    &= \frac{4 \pi}{\tbar} \left(\frac{\pi t^2}{\tbar} + \Theta_0t \right) + 2 t \beta_{0,\Theta} + 2\sum_{\lambda\not=0} Q^*_\lambda(\mbx_0) \frac{\beta_{\lambda,\Theta}}{\lambda} \left(e^{\lambda t} - 1 \right).
\end{aligned}
\end{equation*}
Here we have introduced the coefficients
\begin{equation}\label{eq:beta-theta-coef}
 \beta_{\lambda,\Theta}=\sum_{i,j}\int_{\mbx \notin \mathcal{U}_\epsilon}   P_\lambda(\mbx)  \mathcal{G}_{ij}(\mbx)\left(\frac{\partial \Theta(\mbx)}{\partial x_{i}}\frac{\partial \Theta(\mbx) }{\partial x_{j}} \right) d\mbx, 
\end{equation}
that can be obtained once we know the phase mapping $\Theta(\mbx)$, the stationary density $P_0(\mbx)$, and the rest of the forward eigenfunctions $P_\lambda(\mbx)$.
\\ \\
With the above results, we obtain the variance of $\Theta(\mbX)$
\begin{equation}\label{eq:thetaVar}
\begin{aligned}
    \text{Var}(\Theta(\mbX)) &= \lr{\Theta^2(\mbX)}_{\mbx_0} - \lr{\Theta(\mbX)}^2_{\mbx_0},\\
    &= 2t \beta_{0,\Theta} + 2\sum_{\lambda\neq0} Q^*_\lambda(\mbx_0) \frac{\beta_{\lambda,\Theta}}{\lambda} \left(e^{\lambda t} - 1 \right).
\end{aligned}
\end{equation}

From this result we can obtain the diffusion constant of the MRT phase, capturing the asymptotic rate of spread of the unwrapped phase around its mean value. 
For a given phase function $\phi$ its respective diffusion constant $D^\phi_{\text{eff}}$ is defined as
\begin{equation}
D^\phi_{\text{eff}} = \lim_{t \to \infty} \frac{1}{2t} ~\text{Var}(\phi(t)).
\end{equation}
Therefore for the MRT we find
\begin{equation}\label{eq:phase-diff-cte}
D^\Theta_{\text{eff}} = \lim_{t \to \infty} \frac{1}{2t} ~\text{Var}(\Theta(t)) = \beta_{0,\Theta} = \sum_{i,j} \int_{\mbx \notin \mathcal{U}_\epsilon}   P_0(\mbx)  \mathcal{G}_{ij}(\mbx)\left(\frac{\partial \Theta(\mbx)}{\partial x_{i}}\frac{\partial \Theta(\mbx) }{\partial x_{j}} \right) d\mbx.
\end{equation}
%


\begin{remark}
Given a stochastic system as in \eqref{eq:SDE}, its phase diffusion coefficient is an asymptotic property of the system that should not depend on the specific choice of the phase observable (as long as one considers a proper phase definition). Therefore, our formula \eqref{eq:phase-diff-cte}, which just requires the knowledge of the phase function $\Theta(\mbx)$ and the stationary density $P_0(\mbx)$, provides an exact expression for this asymptotic property for any general phase definition.
\end{remark}

\subsubsection{The Amplitude case}\label{sec:amp-variance}

We can also use the preceding results to compute the variance of the stochastic amplitude $\BSigma(\mbx)$.
Unlike the phase function $\Theta$, the amplitude function is $C^2$ throughout the domain $\mathcal{D}'$ in every example we have encountered.  
Considering \eqref{eq:square-avg} and substituting $\Upsilon^2(\mbX)$ by  $\BSigma^2(\mbX)$ we obtain
\begin{equation*}\small
\begin{aligned}
    \frac{d}{dt}\!\lr{\BSigma^2(\mbX)}_{\mbx_0} \! &= 2 \lr{\BSigma(\mbX)\LLd[\BSigma(\mbX)]}_{\mbx_0} +  2\sum_{i,j} \lr{ \mathcal{G}_{ij}(\mbX)\left(\frac{\partial \BSigma(\mbX)}{\partial x_{i}}\frac{\partial \BSigma(\mbX) }{\partial x_{j}} \right) }_{\mbx_0} \\
    &= 2 \lambda_{\text{Floq}} \lr{\BSigma^2(\mbX)}_{\mbx_0} +  2\sum_{i,j} \int d\mbx ~P(\mbx, t| \mbx_0, 0) ~\mathcal{G}_{ij}(\mbx)\left(\frac{\partial \BSigma(\mbx)}{\partial x_{i}}\frac{\partial \BSigma(\mbx) }{\partial x_{j}} \right) \\
    &=  2 \lambda_{\text{Floq}} \lr{\BSigma^2(\mbX)}_{\mbx_0} \!+  2\sum_{i,j} \!\int\! d\mbx \!\left( \!P_0(\mbx) \!+\! \sum_{\lambda\neq0} Q^*_\lambda(\mbx_0)P_\lambda(\mbx) e^{\lambda t}\!\right) \!\mathcal{G}_{ij}(\mbx)\!\left(\frac{\partial \BSigma(\mbx)}{\partial x_{i}}\frac{\partial \BSigma(\mbx) }{\partial x_{j}} \right) \\
    &=  2 \left(\lambda_{\text{Floq}} \lr{\BSigma^2(\mbX)}_{\mbx_0} + \beta_{0,\BSigma} \right) +  2 \sum_{\lambda\neq0} Q^*_\lambda(\mbx_0) \beta_{\lambda,\BSigma} e^{\lambda t},
\end{aligned}
\end{equation*}\normalsize
where, similarly as in \eqref{eq:beta-theta-coef}, we have introduced the coefficients
\begin{equation}
  \beta_{\lambda,\BSigma}=  \!\int\! d\mbx  \sum_{i,j}  P_\lambda(\mbx)  \mathcal{G}_{ij}(\mbx) \frac{\partial \BSigma(\mbx)}{\partial x_{i}}\frac{\partial \BSigma(\mbx) }{\partial x_{j}}.
\end{equation}
Solving the above  nonhomogeneous linear equation for the variance (with time-dependent forcing) gives
%
\begin{equation}
    \lr{\BSigma^2(\mbX)}_{\mbx_0} =  - \frac{\beta_{0,\BSigma}}{\lambda_{\text{Floq}}} + \left(\BSigma^2_0 + \frac{\beta_{0,\BSigma}}{\lambda_{\text{Floq}}} \right)e^{2t\lambda_{\text{Floq}}} + 2  \sum_{\lambda\neq0} \frac{Q^*_\lambda(\mbx_0) \beta_{\lambda,\BSigma}}{\lambda -2\lambda_{\text{Floq}}} \left( e^{\lambda t} - e^{2\lambda_{\text{Floq}} t} \right). 
\end{equation}
Consequently, we obtain  the variance of $\BSigma$  as
\begin{equation}\label{eq:sigmaVar}
\begin{aligned}
    \text{Var}(\BSigma(\mbX)) &= \lr{\BSigma^2(\mbX)}_{\mbx_0} - \lr{\BSigma(\mbX)}^2_{\mbx_0}\\
    & = - \frac{\beta_{0,\BSigma}}{\lambda_{\text{Floq}}} \left(1 - e^{2t\lambda_{\text{Floq}}}\right) + 2  \sum_{\lambda\neq0} \frac{Q^*_\lambda(\mbx_0) \beta_{\lambda,\BSigma}}{\lambda -2\lambda_{\text{Floq}}} \left( e^{\lambda t} - e^{2\lambda_{\text{Floq}} t} \right).
\end{aligned}
\end{equation}
Thus $\text{Var}(\BSigma(\mbX))$ asymptotically approaches a constant value given by
\begin{equation}\label{eq:st-sigmaVar}
    \lim_{t \to \infty} \text{Var}(\BSigma(\mbX)) = - \frac{\beta_{0,\BSigma}}{\lambda_{\text{Floq}}} = \frac{-1}{\lambda_{\text{Floq}}} \left( \sum_{i,j} \int   P_0(\mbx)  \mathcal{G}_{ij}(\mbx)\left(\frac{\partial \BSigma(\mbx)}{\partial x_{i}}\frac{\partial \BSigma(\mbx) }{\partial x_{j}} \right) d\mbx \right).
\end{equation}
We see that the variance involves a ratio between the expected value of the square of the amplitude diffusive terms in \eqref{eq:stocPhaseAmp} and the amplitude contraction (represented by $|\lambda_{\text{Floq}}|$).\\

\section{Numerical Examples}\label{sec:section6}
	
	Next, we illustrate our method by applying it to systems with a noise-perturbed LC as well as to systems with truly noise-induced oscillations. For the numerical procedure underlying these results, we refer the reader to the Appendix.
	
	\subsection{A spiral sink}\label{ssec:oup}
	
	We start by considering a classical and well studied stochastic process: a two-dimensional Ornstein-Uhlenbeck process (OUP) such that the origin becomes a stable sink \cite{uhlenbeck1930theory, Gar85, leen2016eigenfunctions, thomas2019phase}. 
 The general Langevin equation is:
	\begin{equation}\label{eq:langEq}
		\dot{\mbx} = A\mbx + B\xi,
	\end{equation}
	where we assume that the two eigenvalues of $A$ are a complex conjugate pair denoted as $\lambda_\pm = \mu \pm i \omega$ with $\mu<0$. 
	For concreteness, we choose the matrices $A$ and $B$ as
	\begin{equation}\label{eq:abMatrices}
		A = \begin{pmatrix} \mu & -\omega \\ \omega & \mu \end{pmatrix}
		, \quad B = \begin{pmatrix} B_{11} & B_{12} \\ B_{21} & B_{22} \end{pmatrix}.
	\end{equation}
	Following this notation, we can write the matrix $\mathcal{G} = \frac{1}{2} BB^\top$ in $\LLd$ in the following way \cite{thomas2019phase}
	\begin{equation}
		\mathcal{G}
			= \frac{1}{2} \begin{pmatrix} B^2_{11} + B^2_{12} & B_{11}B_{21} + B_{12}B_{22} \\ B_{11}B_{21} + B_{12}B_{22} & B^2_{22} + B^2_{21} \end{pmatrix}
			= \epsilon \begin{pmatrix} 1 + \beta_D  & \beta_c \\ \beta_c & 1 - \beta_D \end{pmatrix}.
	\end{equation}
	
	To study the OUP, we choose $\mu = -0.1$ and $\omega= 0.5$ and isotropic noise $B=\sqrt{2D}\cdot[1,0;0,1]$ with 
	$D=0.00125$ from which we obtain $\epsilon, \beta_D,$ $\beta_c =  [D, 0, 0]$. For these parameters, the MRT period is given by $\tbar = \frac{2 \pi (\omega^2 + \mu^2(1 - \beta^2_c - \beta^2_D))}{\omega(\mu^2 + \omega^2)} = 4\pi$ (see \cite{perez2022quantitative}).
	
	The mean--return-time phase function of the OUP is given by $\Theta(\mbx) = \arctan(\mby/\mbx)$ (see Fig.~\ref{fig:focusPanel}a). 
 As  shown in \cite{perez2021isostables}, its amplitude function is given by $\Sigma(\mbx) = 2 + \mu(\mbx^2 + \mby^2)/\epsilon$, so there is an effective radius $\Sigma_0 = |2\epsilon/\mu|$ (see Fig.~\ref{fig:focusPanel}b) to which trajectories decay in the mean (see Fig.~\ref{fig:focusPanel}f). From these phase-amplitude functions we can easily obtain the effective vector field (see Fig.~\ref{fig:focusPanel}c)
    \begin{equation*}
        \dot{\mbx} = \mu\mbx - \omega\mby + 2\varepsilon\frac{\mbx}{\mbx^2 + \mby^2}, \qquad \dot{\mby} = \omega\mbx  + \mu\mby + 2\varepsilon\frac{\mby}{\mbx^2 + \mby^2},
    \end{equation*}
    showing how the expansive effects of the noise ($\epsilon > 0$), combine with the dissipative effect of the dynamics ($\mu < 0$) to effectively sustain oscillations.

       As is widely known, the probability density of the OUP is given by a Gaussian function with a maximum at the origin \cite{uhlenbeck1930theory, Gar85}. Therefore, since the MRT function $\Theta(\mbx)$ is known, it is possible to write the integral in \eqref{eq:phase-diff-cte} yielding the diffusion constant of the unwrapped MRT phase 
    \begin{equation}\label{eq:focus-diff-constant}  \small  D^\Theta_{\text{eff}} = \lim_{r_0 \to 0} \frac{1}{ C} \int^{\infty}_{r_0}  \frac{dr}{r}\exp\left[\mu \frac{r^2}{2\epsilon}\right], \quad\enskip \text{where} \enskip C =  2\pi \! \int^{\infty}_{r_0} \!  dr~r\exp\left[\mu \frac{r^2}{2\epsilon}\right] = \frac{2\pi\epsilon}{|\mu|}\exp\left[\mu \frac{r_0^2}{2\epsilon}\right].
    \end{equation}\normalsize
    where we have introduced $r = \sqrt{x^2 + y^2}$. 
    As expected (see Remark~\ref{rm:phase-cte-div}) the integral \eqref{eq:focus-diff-constant} yielding $D^\Theta_{\text{eff}}$ would diverge if taken over the entire domain $r\ge 0$.   
    For this reason, we have written an improper integral in which $r_0>0$ (note $\epsilon\equiv r_0$ in  Remark \ref{rm:phase-cte-div}) is a cut-off radius preventing the integral in \eqref{eq:focus-diff-constant}  diverging. 
    We note that considering this cut-off radius $r_0$, requires to recompute the normalisation constant $C$ of the stationary probability density accordingly. 
    To check the validity of our theoretical approach, we have computed $D^\Theta_{\text{eff}}$ in \eqref{eq:focus-diff-constant} for three different values of $r_0 = [10^{-2}, 10^{-3}, 10^{-4}]$. These theoretical values were compared with the phase diffusion constant obtained by integrating the OUP in \eqref{eq:langEq} with reflecting boundary conditions at the three corresponding values of $r_0$. As shown in Fig.~\ref{fig:focusPanel}g, we find a very good agreement with the predictions of the integral in \eqref{eq:focus-diff-constant} and the results of numerical simulations with reflecting boundary conditions at $r_0$. Regarding the variance, we observe that $\text{Var}(\BSigma(\mbx))$ approaches the expected value of $\beta_{0,\BSigma}$ in the asymptotic limit, $t \to \infty$, in agreement with the results of numerical simulations (see Fig.~\ref{fig:focusPanel}h).


	
	\subsection{A supercritical Hopf bifurcation}
    Next, we study a well known canonical dynamical system: the normal form of a supercritical Hopf bifurcation
    \begin{equation}
	   \begin{aligned}\label{eq:slEDOs}
				\dot{\mbX} &= \beta \mbX - \mbY - \mbX(\mbX^2 + \mbY^2)+ \sqrt{2D_x} \xi_x(t), \\
				\dot{\mbY} &= \mbX + \beta \mbY - \mbY(\mbX^2 + \mbY^2)+ \sqrt{2D_y} \xi_y(t),
			\end{aligned}
		\end{equation}	
		with $\beta = 1$. 
  In the absence of noise, this system shows a limit cycle of period $T = 2\pi$ and characteristic multiplier of $\lambda = -2$. 
  To illustrate  how the proposed framework captures the effects of noise, we will study this model for anisotropic noise ($D_x = 0.1$ and $D_y = 2.5 \cdot 10^{-4}$).

	As Fig.~\ref{fig:hopfPanel} shows, the main effect of anisotropic noise is to break the characteristic rotational symmetry of the deterministic Hopf system. This symmetry breaking is particularly remarkable for the isostable function (Fig.~\ref{fig:hopfPanel}b) and, consequently, it is also reflected in the effective vector field (Fig.~\ref{fig:hopfPanel}c). We observe that transforming the trajectory of the effective vector field to the phase and amplitude functions matches the expected value of the same phase-amplitude functions over trajectories of the SDE \eqref{eq:slEDOs} (Fig.~\ref{fig:hopfPanel} panels d, e, f). 
 Finally, we observe  excellent agreement between  the theoretical predictions for both $\text{Var}(\Theta(\mbx))$ and $\text{Var}(\BSigma(\mbx))$ in formulas \eqref{eq:thetaVar} and \eqref{eq:sigmaVar}, and numerical simulations (see Fig.~\ref{fig:hopfPanel} panels g,h). 
 Note that in contrast with the OUP example in \S\ref{ssec:oup}, for the Hopf normal form \eqref{eq:slEDOs} the stationary probability near the phaseless point at the origin is very small.
 The effect that trajectories passing near this point have on the effective diffusion constant for the unwrapped MRT phase may be neglected.

	%
	\subsection{A saddle node on an invariant cycle (SNIC) bifurcation}
    As a final example  illustrating the phase-amplitude framework we put forward in this manuscript we choose a normal form for a Saddle-Node on an Invariant Circle (SNIC) endowed with white Gaussian noise
    \begin{equation}
	   \begin{aligned}\label{eq:snicEDOs}
				\dot{\mbX} &= \beta \mbX - m\mbY - \mbX(\mbX^2 + \mbY^2) + \frac{\mbY^2}{\sqrt{\mbX^2 + \mbY^2}} + \sqrt{2D_x} \xi_x(t), \\
				\dot{\mbY} &= m\mbX + \beta \mbY - \mbY(\mbX^2 + \mbY^2) -  \frac{\mbX\mbY}{\sqrt{\mbX^2 + \mbY^2}} + \sqrt{2D_y} \xi_y(t).
			\end{aligned}
		\end{equation}	
		
        In the absence of noise, the saddle-node bifurcation from a pair of fixed points to a LC of radius $\sqrt{\beta}$ occurs at $m = 1$ (see \cite{guillamon2009computational} for more details). 
        Here we will consider $\beta=1, m=0.99 $ and $D_x=D_y=0.25$. 
        For this noise level, even if we set parameters for which there is not an underlying LC, the presence of noise leads to the appearance of noise induced oscillations. 
        We regard this regime as excitable \cite{lindner2004effects} since the mechanism sustaining oscillations involves the noise pushing trajectories past the stable manifold of the deterministic saddle point, from whence they perform a  large excursion, then rapidly return to the stable deterministic node.
        
        As we observe in Fig.~\ref{fig:snicPanel}a, the MRT phase function $\Theta(\mbx)$ presents the characteristic structure of a saddle-node bifurcation. 
        Near the ghost of the saddle-node point, the isochrons are more densely packed. The stochastic amplitude  function in Fig.~\ref{fig:snicPanel}b shows its zeroth level set $\Sigma_0$ close to the invariant curve (a circle of radius $\sqrt{\beta}$). 
        If we compare the deterministic vector field and the effective vector field (Fig.~\ref{fig:snicPanel}c), we see that, even if they are very similar, the effective vector field shows oscillations while the deterministic vector field does not. 
        Moreover, as ensured by  Theorem~\ref{thm:bigthm}, transforming  the trajectory of the effective vector field to the phase and amplitude functions agrees with the averaged value of the same phase-amplitude functions over trajectories of the SDE \eqref{eq:snicEDOs} (Fig.~\ref{fig:snicPanel} panels d, e, f). 
        We also observe  good agreement between the theoretical formulas in \eqref{eq:thetaVar} and \eqref{eq:sigmaVar} for both $\text{Var}(\Theta(\mbx))$ and $\text{Var}(\BSigma(\mbx))$ and the numerical simulations (see Fig.~\ref{fig:snicPanel} panels g,h).
        As in the case of the Stuart-Landau system, the phaseless point occurs in a region of low probability density, and its effect on the effective diffusion constant of the unwrapped MRT phase may be neglected.
	
\begin{figure}[H]
	\centering\includegraphics[width=1\textwidth]{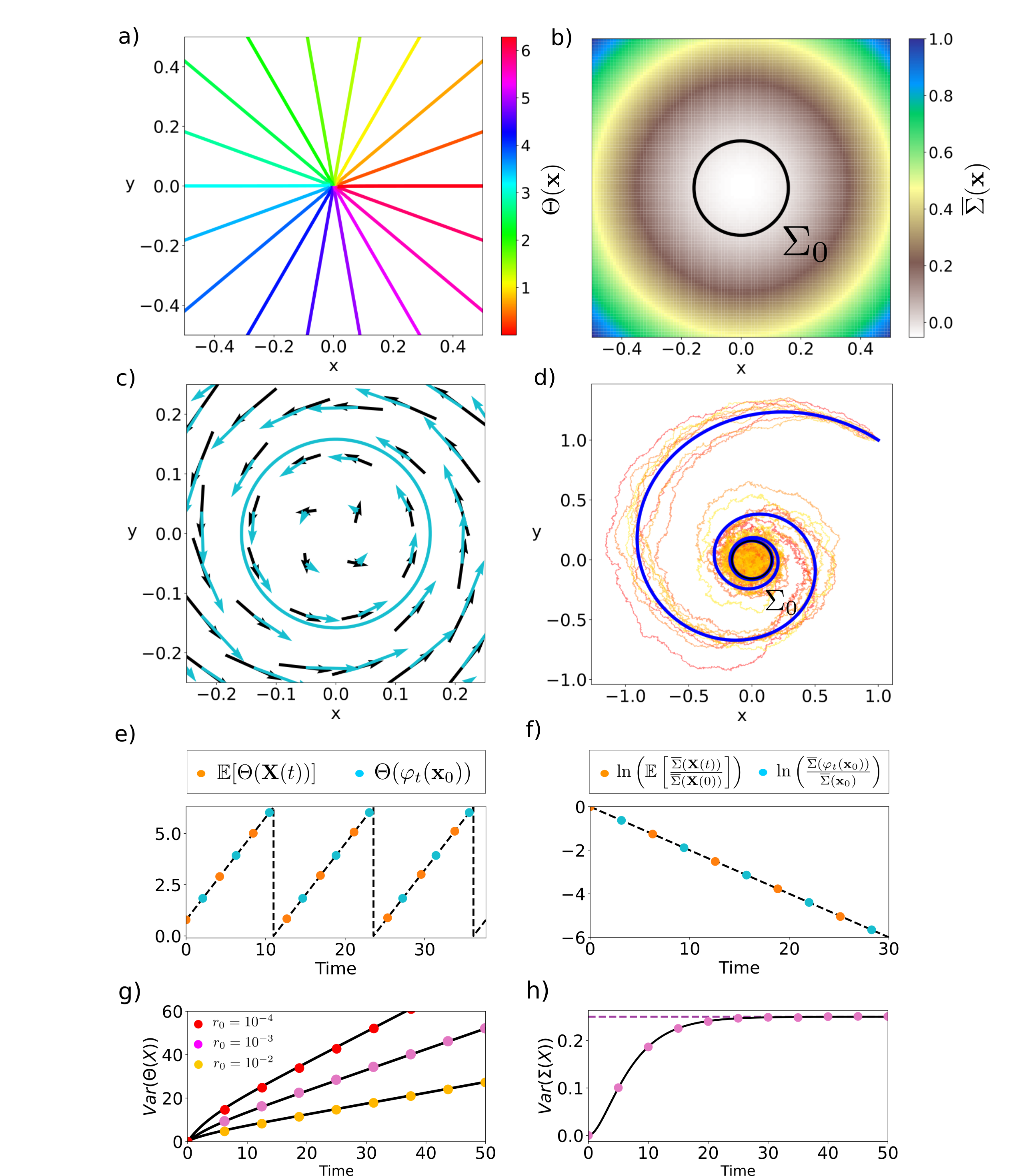}
	\caption{Noisy Spiral Sink in \eqref{eq:langEq} for $\mu = -0.1,~\omega = 0.5,~D = 0.00125$. For these parameters we find $\tbar = 4\pi$ and $\lambda_{\text{Floq}} = 2\mu = -0.2$. (a-b) Mean-Return-Phase $\Theta(\mbx)$ function and Isostable function $\BSigma(\mbx)$ ($\mby$-axes shared). The black curve in (b) corresponds to the set $\Sigma_0 = \{ \mbx ~|~ \BSigma=0 \}$. (c) Deterministic vector field $\mbf(\mbx)$ (black) and effective vector field $\mbF(\mbx)$ (blue). We also plot $\Sigma_0$ in blue. (d) Ten trajectories $\mbX(t)$ generated by the SDE \eqref{eq:langEq} (yellow to red colors) and the trajectory $\varphi_t(\mbx_0)$ corresponding to the effective vector field (blue). 
 Note the different scales in (c) and (d) compared to (a-b).
 (e) $\E[\Theta(\mbX(t))]_{\mbx_0}$ and $\Theta(\varphi_t(\mbx_0))$ versus time, showing the predicted linear growth (black dashed line). (f) $\ln(\E\left[\BSigma(\mbX(t))/\BSigma(\mbX(0))\right]_{\mbx_0})$ and $\ln\left(\BSigma(\varphi_t(\mbx_0))/\BSigma(\mbx_0)\right)$ versus time, showing the predicted exponential decay (black dashed line). (g-h) Variance of the Mean-Return-Phase for three different cut-off radius $r_0$ and the Isostable function, respectively. Black curve theory, dots computations. }\label{fig:focusPanel}
\end{figure}

\begin{figure}[H]
	\centering\includegraphics[width=1\textwidth]{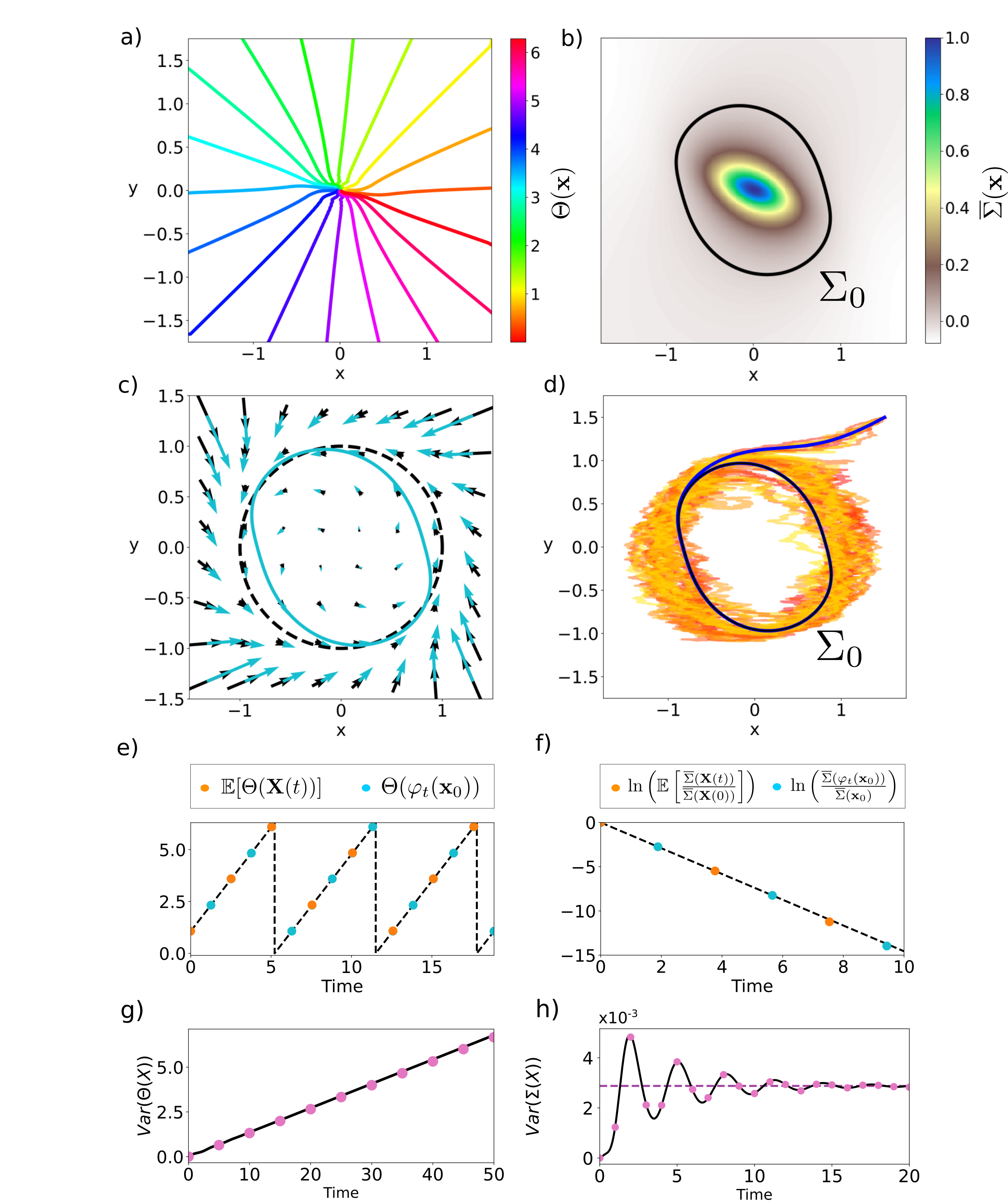}
	\caption{Noisy supercritical Hopf bifurcation in  \eqref{eq:slEDOs} for $\beta = 1$ and anisotropic noise $D_x = 0.1$ and $D_y = 2.5 \cdot 10^{-4}$. For these parameters we find $\tbar = 6.288$ and $\lambda_{\text{Floq}} = -1.456$. (a-b) Mean-Return-Phase $\Theta(\mbx)$ function and Isostable function $\BSigma(\mbx)$ ($\mby$-axes shared). The black curve in (b) corresponds to the set $\Sigma_0 = \{ \mbx ~|~ \BSigma=0 \}$. (c) Deterministic vector field $\mbf(\mbx)$ (black) and effective vector field $\mbF(\mbx)$ (blue). We also plot $\Sigma_0$ in blue. (d) Ten trajectories $\mbX(t)$ generated by the SDE \eqref{eq:slEDOs} (yellow to red colors) and the trajectory $\varphi_t(\mbx_0)$ corresponding to the effective vector field (blue). (e) $\E[\Theta(\mbX(t))]_{\mbx_0}$ and $\Theta(\varphi_t(\mbx_0))$ versus time, showing the predicted linear growth (black dashed line). (f) $\ln(\E\left[\BSigma(\mbX(t))/\BSigma(\mbX(0))\right]_{\mbx_0})$ and $\ln\left(\BSigma(\varphi_t(\mbx_0))/\BSigma(\mbx_0)\right)$ versus time, showing the predicted exponential decay (black dashed line). (g-h) Variance of the Mean-Return-Phase and the Isostable function, respectively. Black curve theory, purple dots computations.}\label{fig:hopfPanel}
\end{figure}

\begin{figure}[H]
	\centering\includegraphics[width=1\textwidth]{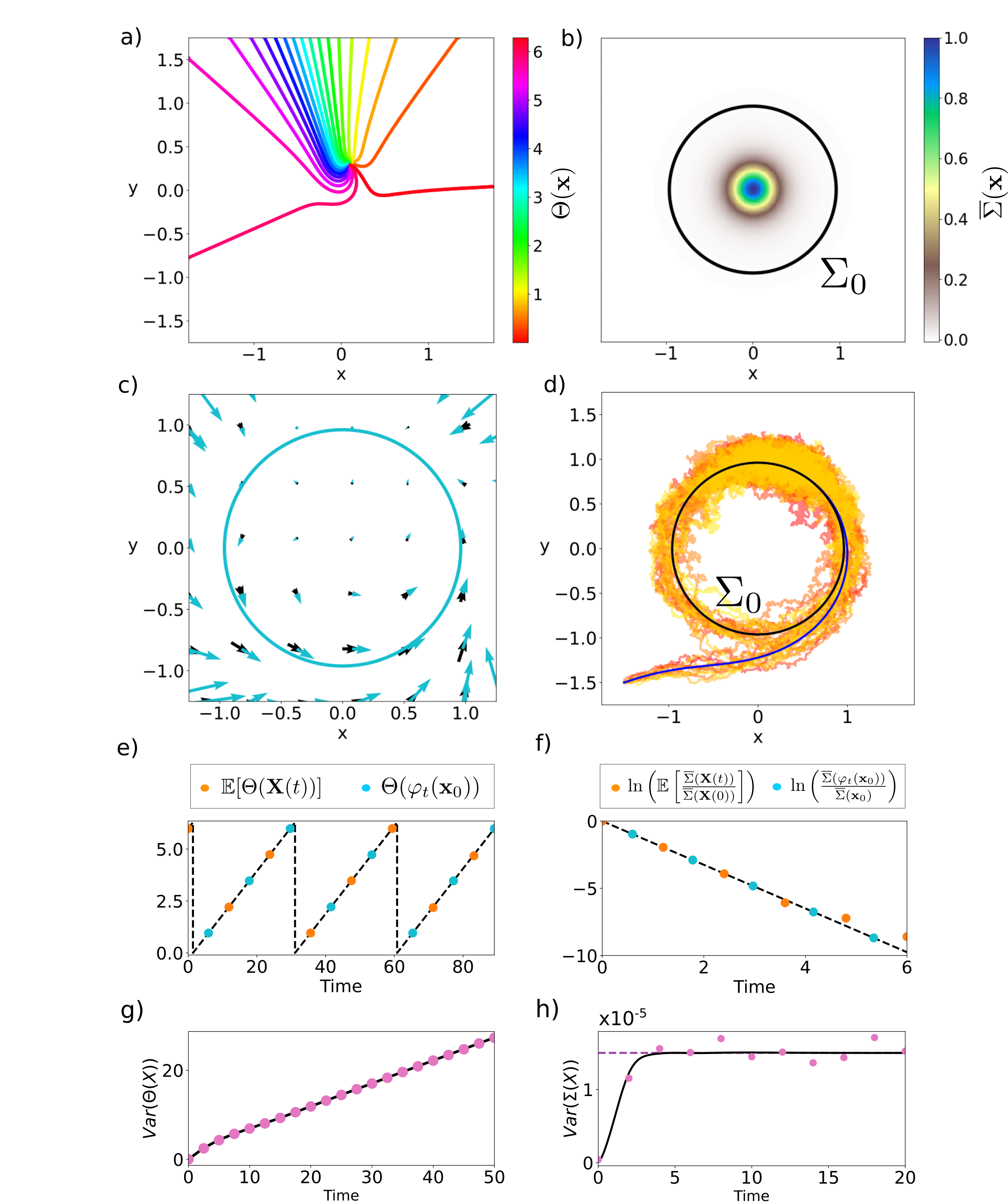}
	\caption{Noisy SNIC bifurcation in  \eqref{eq:snicEDOs} for $\beta = 1, m=0.99$ and $D_x, D_y = \sqrt{2D}\cdot[1,1]$ with $D=0.025$. For these parameters we find $\tbar = 29.696$ and $\lambda_{\text{Floq}} = -1.625$. (a-b) Mean-Return-Phase $\Theta(\mbx)$ function and Isostable function $\BSigma(\mbx)$ ($\mby$-axes shared). The black curve in (b) corresponds to the set $\Sigma_0 = \{ \mbx ~|~ \BSigma=0 \}$. (c) Deterministic vector field $\mbf(\mbx)$ (black) and effective vector field $\mbF(\mbx)$ (blue). We also plot $\Sigma_0$ in blue. (d) Ten trajectories $\mbX(t)$ generated by the SDE \eqref{eq:snicEDOs}(yellow to red colors) and the trajectory $\varphi_t(\mbx_0)$ corresponding to the effective vector field (blue). (e) $\E[\Theta(\mbX(t))]_{\mbx_0}$ and $\Theta(\varphi_t(\mbx_0))$ versus time, showing the predicted linear growth (black dashed line). (f) $\ln(\E\left[\BSigma(\mbX(t))/\BSigma(\mbX(0))\right]_{\mbx_0})$ and $\ln\left(\BSigma(\varphi_t(\mbx_0))/\BSigma(\mbx_0)\right)$ versus time, showing the predicted exponential decay (black dashed line). (g-h) Variance of the Mean-Return-Phase and the Isostable function, respectively. Black curve theory, purple dots computations.}\label{fig:snicPanel}
\end{figure}

\newpage
	
\section{Discussion}\label{sec:discussion} 
The classical Parameterization Method (PM) uses a change of coordinates to transform a dynamical system, for instance a limit cycle system, to a representation in which the transformed variables evolve as simply as possible, viz.~either at constant or linear rates.  
In this paper we have proposed an extension of the PM to stochastic oscillators.
Specifically, we have assembled previously proposed phase and amplitude functions for stochastic limit cycles into a unified change of coordinates that, in the mean, evolve as simply as possible.
In addition, we have established several novel results concerning this phase-amplitude description for noisy oscillators.
We have proven the existence and uniqueness of an effective vector field whose orbits capture fundamental properties of the stochastic limit cycle.  
And we have derived expressions for the variance of the phase and isostable functions over both short and long time intervals.  

Our approach is inspired by the key idea of the Parameterization Method: given an SDE as in \eqref{eq:SDE}, one seeks the phase-amplitude transformation for which the resulting dynamics are expressed as simply as possible. 
Since the change of variables for stochastic systems follows the It\^{o} rule, the deterministic phase-amplitude variables no longer evolve in a simple way. 
Moreover, the definition of the deterministic phase-amplitude functions is premised upon the existence of a limit cycle (LC).
A noise-perturbed LC is but one of several distinct mechanisms which may underlie the generation of stochastic oscillations \cite{perez2023universal}. 
These drawbacks require, in our view, the development of new phase-amplitude functions. 
As we have argued in the manuscript, the phase and amplitude functions overcoming these limitations correspond to two already known functions: the Mean--Return-Time (MRT) phase \cite{schwabedal2013phase,cao2020partial} and the stochastic amplitude \cite{perez2021isostables}. 
The existence of these two functions does not depend on the existence of an underlying deterministic limit cycle, but only on mild properties of the Kolmogorov backwards operator.\footnote{Also known as the generator of the Markov process, and as the stochastic Koopman operator \cite{mezic2005spectral,vcrnjaric2019koopman,perez2023universal}.} 
Moreover, we found that transforming to these two variables yielded drift terms that evolve in a similar manner as their deterministic analogues, thus satisfying our goal of rendering the dynamics ``as simply as possible."



Besides assembling in a different light previous results in the literature, in this paper we also presented new results obtained by means of this phase-amplitude construction. 
The idea of finding a vector field capturing how the noise \textit{effectively} modifies the underlying dynamics of the system to sustain oscillations was previously put forward in \cite{perez2021isostables}. 
In the present paper, we not only gave a precise definition, but also proved in Theorem~\ref{thm:bigthm} the existence and uniqueness of the vector field with the properties we defined. 
Moreover, we also presented new results on the higher moments of both the MRT and the isostable function (Eqs.~\eqref{eq:thetaVar} and \eqref{eq:sigmaVar}, respectively). 
From these formulas it was straightforward to obtain the effective diffusion constant of the unwrapped MRT phase (Eq.~\eqref{eq:phase-diff-cte}) and the stationary variance of the stochastic amplitude (Eq.~\eqref{eq:st-sigmaVar}). We have illustrated the meaning of the effective vector field and checked the validity of our results for three qualitatively different stochastic oscillators: a noisy spiral sink, a noisy LC and an excitable system.

\paragraph{Extension to higher dimensions}  
In the analysis of stochastic processes, one could view the construction of a martingale \cite{DitlevsenGreenwood2013JM,Williams1991probability} as playing a simplifying role comparable to the construction of phase-amplitude coordinates in a deterministic LC system.  
Indeed, after subtracting the mean rate of change of the unwrapped MRT phase, one is left with a balanced fluctuation, i.e.~$\E\left[d\left(\Theta-2\pi t/\tbar\right)\right]\equiv 0$, which is a martingale.
Similarly, after subtracting the mean decay rate of the isostable function, one is again left with a mean-zero fluctuation,
so that $\E\left[d\left(\BSigma(t)-e^{\lambda_{\text{Floq}} t}\BSigma(0)\right)\right]\equiv 0$ is also a martingale.
Thus, for planar systems, we have constructed two separate martingales that together give a complete description of the mean dynamics.
The extension of this framework to higher dimensional ($d\ge 3$) systems would require construction of additional martingales distinct from these first two.  
In the deterministic setting, the amplitude coordinates corresponding to additional Floquet modes provide the additional variables needed for a complete description \cite{perez2020global}.
In the stochastic setting, one may conjecture that the real eigenfunctions corresponding to more rapidly decaying modes may play a similar role, provided their eigenvalues are not just integer multiples of the leading negative eigenvalue.  
Alternatively, higher complex modes may provide the needed coordinates.  
But delicate issues arise.  
We note that in some cases, the leading complex eigenvalues $\mu\pm i\omega$ may be associated with a parabolic family of associated eigenvalues $k^2\mu\pm i k \omega$, cf.~Fig.~3 of \cite{thomas2014asymptotic}.
This set of eigenvalues would form a single ``family", the associated eigenfunctions of which would not necessarily provide suitable additional coordinates.  
In other cases, there may be more than one set of surfaces satisfying the mean--return-time criterion, with different mean periods \cite{thomas2015comment,ThomasLindner2015Reply-PhysRevLett.115.069402}.  
Many interesting open questions remain concerning the extension of the phase-amplitude framework to stochastic oscillators in higher dimensions.

\section*{Acknowledgments}
This work was supported in part  
by the Oberlin College Department of Mathematics.

\bibliographystyle{siamplain}
\bibliography{references}

\newpage

\appendix

\section{Numerical Details}

To generate the numerical results in the main text, we followed the procedure in \cite{cao2020partial, perez2023universal} (see also \cite{perez2021isostables} and \cite{perez2022quantitative}). Since the phase space for all the two-dimensional examples considered in this paper is unbounded, we consider a (finite) rectangular domain $\mathcal{X}$
\begin{equation}\label{eq:domainD_AP}
	\mathcal{X} = [ x_1^-, x_1^+ ] \times [x_2^-, x_2^+].   
\end{equation} 
whose size is chosen large enough so that the probability for individual trajectories $\mbx(t)$ to reach the boundaries is very low. Then, we just need to discretise the domain $\mathcal{X}$ in $N$ and $M$ points such that $\Delta x_1 = (x_{1}^+ - x_1^-)/N$ 
and $\Delta x_2 = (x_2^+ - x_2^-)/M$, to build $\LLd$ (and/or $\LL$) by using a standard finite-difference scheme. In general, we used centered finite differences, for instance, given a function $T(x_1, x_2)$
\begin{equation}
	(\partial_{x_1} T)_{i,j} = \frac{T_{i+1,j} - T_{i-1,j}}{2\Delta x_1}, \qquad  (\partial_{x_1x_1} T )_{i,j} = \frac{T_{i+1,j} - 2T_{i,j} + T_{i-1,j}}{(\Delta x_1)^2},
\end{equation}
and, as for boundary conditions, we used adjoint reflecting boundary conditions
\begin{equation}
    \sum_{j=1,2} n_j \sum_{k=1,2} \mathcal{G}_{jk} \partial_{x_k} T(x_1,x_2) = 0,
\end{equation}
where $n$ is the local unit normal vector at $\mathcal{X}$ boundaries and $\mathcal{G} = \frac12 \mbg\mbg^\intercal$ (see \eqref{eq:forward-eq}).


After diagonalizing the resulting $(N \cdot M, N \cdot M)$ matrix, we obtain the eigenvalues and the associated eigenfunctions of $\LLd$ ($\LL$). We recall that we are not interested in the complete spectrum of $\LLd$ ($\LL$). For $\LLd$ we just consider (and hence present in Fig.~\ref{fig:eigenvals}) the part of the spectrum which is relevant for our analysis. That is, we consider mainly the eigenvalue associated with the slowest decaying \textit{complex} eigenfunction $Q^*_1(\mbx)$, from which we get the stochastic asymptotic phase $\psi(\mbx) = \arg(Q^*_1(\mbx))$ and slowest decaying \textit{real} eigenfunction $\BSigma(\mbx)$, from which we get the stochastic asymptotic amplitude. Besides these prominent modes (and its corresponding forward modes) we also considered a few more pairs of backward and forward eigenmodes  for computing the transient behaviour of the variance equations \eqref{eq:thetaVar} and \eqref{eq:sigmaVar}.

Following the criterion we used in \cite{perez2023universal}, we normalised each pair of eigenfunctions $Q^*_\lambda(\mbx)$ and $ P_\lambda(\mbx)$  such that they satisfy the following conditions
\begin{equation}\label{eq:norm-conditions-AP}
     \int d\mbx~ |Q^*_\lambda(\mbx)|^2 P_0(\mbx) = 1, \qquad \qquad \int d\mbx~Q^*_{\lambda'}(\mbx)P_\lambda(\mbx)  =  \delta_{\lambda'\lambda},
\end{equation}
so, while the left integral fixes the normalisation of $Q^*_\lambda(\mbx)$ up to an arbitrary complex factor, the second integral fixes the norm and phase of $P_\lambda(\mbx)$.

\section{Obtaining the MRT phase $\Theta(\mbx)$ and $\tbar$}

In this paper, we have followed the procedure we used in \cite{perez2022quantitative} to obtain the MRT phase $\Theta$ by finding a state dependent shift $\Delta \psi(\mbx)$ to the stochastic asymptotic phase $\psi(\mbx)$. We recall we were obtaining $\psi(\mbx)$ as the argument of the Kolmogorov backwards operator complex eigenfunction with least negative real part. That is, having $\psi(\mbx)$ we look for the correction $\Delta \psi(\mbx)$ such that
\begin{equation*}
    \Theta(\mbx) = \psi(\mbx) + \Delta \psi(\mbx)
\end{equation*}
For the specific details about the computations leading to $\Delta \psi(\mbx)$ we refer the reader to \cite{perez2022quantitative}. 

As stated in the manuscript in \cite{cao2020partial} it is possible to find the sufficient conditions guaranteeing the existence of a well-defined mean--return-time period $\tbar$. In addition in this very same manuscript (and also in \cite{perez2022quantitative}) it is explained how to compute $\tbar$.

\section{Numerical Values}

In this section we gave all the relevant numerical values we used to generate the results in the manuscript. We first show the following Table~\ref{tab:table1} with the necessary values for reproducing the numerical discretisations of $\LLd (\LL)$, the obtained leading complex and real eigenvalues and the MRT period $\tbar$
\begin{table}[htbp]\label{tab:table1}
\footnotesize
    \centering
    \caption{Numerical values of the $\LLd (\LL)$ discretisations, the resulting leading complex ($\lambda_1$) and real ($\lambda_{\text{Floq}}$) eigenvalues and the MRT period $(\tbar)$ for the different stochastic oscillators.}
\begin{center}
	\begin{tabular}{cccccccccc}
		&$N$ &$M$ &$x_+$&
		$x_-$ &$y_+$ &$y_-$ &$\lambda_1$  &$\lambda_\text{Floq}$& $\tbar$\\
		\hline
		Sp. Sink & 250 & 250 & 0.75 & -0.75 & 0.75 & -0.75 & -0.1 + 0.5i & -0.2 & 4$\pi$ \\
		Hopf & 250 & 250 & 1.75 & -1.75 & 1.75 & -1.75 & -0.061+1.002i & -1.456 & 6.287 \\
		Excitable SNIC & 250 & 250 & 1.75 & -1.75 & 1.75 & -1.75 & -0.22 + 0.33i & -1.625 & 29.696 
	\end{tabular}
\end{center}
\end{table}

Next we present the values we obtained for computing the variances $\text{Var}(\Theta(\mbX))$ and $\text{Var}(\BSigma(\mbX))$ in \eqref{eq:thetaVar} and \eqref{eq:sigmaVar} respectively. We first define the following notation
\begin{equation}
\begin{aligned}
\alpha_{\lambda, \Theta} &= Q^*_\lambda(\mbx_0)\sum_{i,j} \int   P_\lambda(\mbx)  \mathcal{G}_{ij}(\mbx)\left(\frac{\partial \Theta(\mbx)}{\partial x_{i}}\frac{\partial \Theta(\mbx) }{\partial x_{j}} \right) d\mbx = Q^*_\lambda(\mbx_0) \beta_{\lambda, \Theta},\\
\alpha_{\lambda, \BSigma} &=  Q^*_\lambda(\mbx_0)\sum_{i,j} \int   P_\lambda(\mbx)  \mathcal{G}_{ij}(\mbx)\left(\frac{\partial \BSigma(\mbx)}{\partial x_{i}}\frac{\partial \BSigma(\mbx) }{\partial x_{j}} \right) d\mbx = Q^*_\lambda(\mbx_0) \beta_{\lambda, \BSigma},
\end{aligned}
\end{equation}
where we recall that $Q_0(\mbx) \equiv 1$ as a consequence of the normalisation condition \eqref{eq:norm-conditions-AP}-right. For these reason $\alpha_{0, \Theta} = \beta_{0, \Theta}$ and similarly $\alpha_{0, \BSigma} = \beta_{0, \BSigma}$. Next we discuss the values for each of the considered models.\\

\textbf{Spiral Sink} - For computing the curve $\text{Var}(\Theta(\mbX))$ in panel \ref{fig:focusPanel}, we  considered an initial value for the simulations of $\mbx_0 = (0.05, 0.05)$ and performed our simulations using three different cut-off radius $r_0 = [10^{-2}, 10^{-3}, 10^{-4}]$. For the three considered values of $r_0$, we computed the transient behaviour by taking into account the following set of eigenvalues $\Lambda = \{\lambda_{\text{Floq}} = -0.2, ~\lambda_{\text{Floq},2} = -0.4, ~\lambda_{\text{Floq},3} = -0.6\}$. Next we indicate the values of $\alpha_{\lambda, \Theta}$ for each value of $r_0$ considered. For $r_0 = 10^{-2}$, we obtained $\beta_{0, \Theta} = 0.248$, $\alpha_{\lambda_{\text{Floq}}, \Theta} = 0.158,~ \alpha_{\lambda_{\text{Floq},2}, \Theta} = 0.107,~ \alpha_{\lambda_{\text{Floq},3}, \Theta} = 0.0712$. For $r_0 = 10^{-3}$, we obtained $\beta_{0, \Theta} = 0.4775$, $\alpha_{\lambda_{\text{Floq}}, \Theta} = 0.342,~ \alpha_{\lambda_{\text{Floq},2}, \Theta} = 0.249,~ \alpha_{\lambda_{\text{Floq},3}, \Theta} = 0.177$. Finally, for $r_0 = 10^{-4}$, we obtained $\beta_{0, \Theta} = 0.7077$, $\alpha_{\lambda_{\text{Floq}}, \Theta} = 0.526,~ \alpha_{\lambda_{\text{Floq},2}, \Theta} = 0.392,~ \alpha_{\lambda_{\text{Floq},3}, \Theta} = 0.283$.

For computing the curve $\text{Var}(\BSigma(\mbX))$ in panel \ref{fig:focusPanel}, we obtained $\beta_{0, \BSigma} = 0.05$. In addition, we computed the transient behaviour by taking into account the following set of eigenvalues $\Lambda = \{\lambda_{\text{Floq}} = -0.2, ~\lambda_{\text{Floq},2} = -0.4\}$ for which we computed then, $\alpha_{\lambda_{\text{Floq}}, \BSigma} = -0.05,~ \alpha_{\lambda_{\text{Floq},2}, \BSigma} = -2.8\cdot10^{-5}$. In this case, the simulations considered an initial value of $\mbx_0 = (0.0, 0.0)$.\\

\textbf{Supercritical Hopf} - For computing the curve $\text{Var}(\Theta(\mbX))$ in panel \ref{fig:hopfPanel}, we obtained $\beta_{0, \Theta} = 0.068$. In addition, we did not found a relevant change in the transient when adding different eigenvalues. We considered an initial value for the simulations $\mbx_0 = (0, 1)$.

For computing the curve $\text{Var}(\BSigma(\mbX))$ in panel \ref{fig:hopfPanel}, we obtained $\beta_{0, \BSigma} = 0.0042$. In addition, we computed the transient behaviour by taking into account the following set of eigenvalues $\Lambda = \{\lambda_{2} = -0.22+2.01i,~ \lambda_{4} = -0.79+4.05i,~ \lambda_{\text{Floq}} = -1.459\}$ for which we computed then, $\alpha_{\lambda_{2}, \BSigma} = -0.00245+0.00055i,~ \alpha_{\lambda_{4}, \BSigma} = 0.0009-0.0005i,~  \alpha_{\lambda_{\text{Floq}}, \BSigma} = -0.0011$. In this case, the simulations considered an initial value of $\mbx_0 = (0, 1)$.\\

\begin{figure}[t]
	\centering\includegraphics[width=1\textwidth]{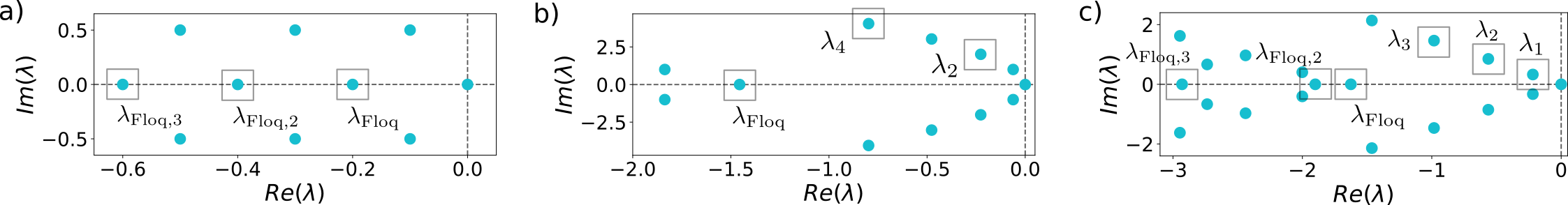}
	\caption{Eigenvalue spectra for the three considered models. a) The spiral sink. b) Supercritical Hopf bifurcation. c) The excitable SNIC. In all cases we highlight the eigenvalues we used for computing the transient behaviour of $\text{Var}(\Theta(\mbX))$ and $\text{Var}(\BSigma(\mbX))$ (see text for more details). }\label{fig:eigenvals}
\end{figure}

\textbf{Excitable SNIC} - For computing the curve $\text{Var}(\Theta(\mbX))$ in panel \ref{fig:snicPanel}, we obtained $\beta_{0, \Theta} = 0.257$. In addition, we computed the transient behaviour by taking into account the following set of eigenvalues $\Lambda = \{\lambda_{1} = -0.22+0.33i, ~\lambda_{2} = -0.56+0.85i, ~\lambda_{3} = -0.98+1.46i\}$ for which we computed then, $\alpha_{\lambda_{1}, \Theta} = 0.098-0.105i,~ \alpha_{\lambda_{2}, \Theta} = 0.017-0.04i,~ \alpha_{\lambda_{3}, \Theta} = 0.0052-0.011i$. We considered an initial value for the simulations of $\mbx_0 = (0, 1)$.

For computing the curve $\text{Var}(\BSigma(\mbX))$ in panel \ref{fig:snicPanel}, we obtained $\beta_{0, \BSigma} = 2.44\cdot10^{-5}$. In addition, we computed the transient behaviour by taking into account the following set of eigenvalues $\Lambda = \{\lambda_{\text{Floq}} = -1.625, ~\lambda_{\text{Floq},2} = -1.9, ~\lambda_{\text{Floq},3} = -2.93 \}$ and hence $\alpha_{\lambda_{\text{Floq}}, \BSigma} = -1.13\cdot10^{-4}, \alpha_{\lambda_{\text{Floq},2}, \BSigma} = 7.89\cdot10^{-5}, \alpha_{\lambda_{\text{Floq},3}, \BSigma} = 1.42\cdot10^{-5}$. In this case, the simulations considered an initial value of $\mbx_0 = (0, 1)$.

\end{document}